\documentclass[leqno]{amsart}
\usepackage[utf8]{inputenc}

\usepackage{graphicx}
\usepackage{color}
\usepackage[all]{xy}
\usepackage{amsfonts}
\usepackage{amsmath,mathtools}
\usepackage{amssymb}
\usepackage{amsthm}
\usepackage[thinlines]{easytable}
\usepackage{array}
\usepackage{enumerate}
\usepackage[hidelinks]{hyperref}
\usepackage[foot]{amsaddr}

\newcommand{\R}{\mathbb{R}}
\newcommand{\C}{\mathbb{C}}
\newcommand{\F}{\mathbb{F}}

\newcommand{\m}{\mathfrak{m}}
\newcommand{\n}{\mathfrak{n}}

\newcommand{\h}{\mathfrak{h}}
\newcommand{\g}{\mathfrak{g}}

\renewcommand{\t}{\mathfrak{t}}
\renewcommand{\k}{\mathfrak{k}}
\newcommand{\T}{\Theta}

\newcommand{\ad}{\mathrm{ad}}

\newcolumntype{M}[1]{>{\centering\arraybackslash}m{#1}}
\newcolumntype{N}{@{}m{0pt}@{}}

\def\squarebox#1{\hbox to #1{\hfill\vbox to #1{\vfill}}}
\def\rm#1{\mathrm{#1}}
\def\cal#1{\mathcal{#1}}
\def\bb#1{\mathbb{#1}}
\def\lie#1{\mathfrak{#1}}

\newtheorem{teoremaA}{Theorem}
\newtheorem{teorema}{Theorem}[section]
\newtheorem{lema}[teorema]{Lemma}
\newtheorem{corolario}[teorema]{Corollary}
\newtheorem{proposicao}[teorema]{Proposition}
\newtheorem{defi}[teorema]{Definition}

\theoremstyle{definition}
\newtheorem{exemplo}[teorema]{Example}
\newtheorem{remark}[teorema]{Remark}

\newcommand{\rank}{\mathrm{rank}}
\newcommand{\Ric}{\mathrm{Ric}}
\newcommand{\mrk}{\mathrm{mrk}}
\newcommand{\vol}{\mathrm{vol}}
\newcommand{\tr}{\mathrm{tr}}
\newcommand{\Ad}{\mathrm{Ad}}

\usepackage{xargs}                 
\usepackage[pdftex,dvipsnames]{xcolor} 

\usepackage[disable,colorinlistoftodos,prependcaption,textsize=tiny]{todonotes}
\newcommandx{\duvida}[2][1=]{\todo[linecolor=red,backgroundcolor=red!25,bordercolor=red,#1]{#2}}
\newcommandx{\comentario}[2][1=]{\todo[linecolor=blue,backgroundcolor=blue!25,bordercolor=blue,#1]{#2}}
\newcommandx{\info}[2][1=]{\todo[linecolor=OliveGreen,backgroundcolor=OliveGreen!25,bordercolor=OliveGreen,#1]{#2}}
\newcommandx{\melhorar}[2][1=]{\todo[linecolor=Plum,backgroundcolor=Plum!25,bordercolor=Plum,#1]{#2}}
\newcommandx{\apagar}[2][1=]{\todo[disable,#1]{#2}}
\begin{document}

\title[The projected Ricci flow and 3-isotropy-summands flag manifolds]{
The projected homogeneous Ricci flow and three-isotropy-summands flag manifolds
}
\author[L. Grama]{Lino Grama$^\dagger$}
\author[R. M. Martins]{Ricardo M. Martins$^\dagger$}
\address{$^\dagger$State University of Campinas -- IMECC, 
	Rua Sérgio Buarque de Holanda, 651
	13083-859, Campinas, SP, Brasil}
\email{lino@ime.unicamp.br, rmiranda@unicamp.br}
\author[M. Patr\~ao]{Mauro Patr\~{a}o$^*$}
\author[L. Seco]{Lucas Seco$^*$}
\address{$^*$University of Brasília -- Departamento de Matemática,	Asa Norte
	70910-900 - Brasilia, DF - Brasil}
\email{mpatrao@mat.unb.br, lseco@unb.br}
\author[L. D. Speran\c{c}a]{Llohann Speran\c{c}a$^\ddagger$}
\address{$^\ddagger$Universidade Federal de São Paulo -- Instituto de Ci\^encia e Tecnologia, Avenida Cesare Mansueto Giulio Lattes, 1201,  12247-014, S\~ao Jos\'e dos Campos, SP, Brazil}
\email{speranca@unifesp.br}
\maketitle

\begin{abstract}
The Ricci flow was introduced by Hamilton and gained its importance through the years. Of special importance is the limiting behavior of the flow and its symmetry properties. Taking this into account, we present a novel normalization for the homogeneous Ricci flow with natural compactness properties. In addition, we present a characterization for Gromov-Hausdorff limits of homogeneous spaces.

As a result, we present a detailed picture of the homogeneous Ricci flow for three-isotropy-summands flag manifolds: phase portraits, basins of attractions, conjugation classes and collapsing phenomena. Moreover, we achieve a full classification of the possible Gromov-Hausdorff limits of the aforementioned lines of flow.
\end{abstract}

\bigskip


\bigbreak


\section{Introduction}

The Ricci flow was introduced by Hamilton \cite{hamilton} and gained its importance through the years. Of special importance is the limiting behavior of the flow and its symmetry properties. Despite its many geometric properties, explicit examples of the flow are not common. On the other hand, homogenous manifolds, particularly flag manifolds, have been a common ground for explicit examples, including results related to the Ricci flow. 

Taking this into account, we present a novel normalization for the homogeneous Ricci flow with natural compactness properties. Our method consists in appropriately normalizing the flow to a simplex and time reparametrizing it to get polynomial equations, obtaining what we call the {\em projected Ricci flow}.  This arises from a natural generalization of the standard unit-volume reparametrization of the Ricci flow and can potentially be applied to the study of the dynamics of other geometric flows on homogeneous spaces.

Apart from the dynamical interest, the paper studies the global geometric behavior of the flow by classifying its limiting manifolds. To this aim, we classify all possible Gromov-Hausdorff limits of families of invariant metrics in a fixed homogeneous space, by proving that they only depend on the  limiting (possibly degenerate) metric, 
which improves a previous collapse result of \cite{buzano}.

As an application of these tools, we present a detailed picture of the homogeneous Ricci flow for three isotropy-summands flag manifolds: phase portraits, basins of attractions, conjugation classes and collapsing phenomena (results summarized below; phase portraits given in figures \ref{sing-su-n} and \ref{gen-ricci-flow-exep}).
We remark that we compute all the possible Gromov-Hausdorff limits of the aforementioned lines of flow. Our global results improve previous results on the dynamics of the Ricci flow for two and three-isotropy summands flag manifolds: \cite{anastassiou-chrysikos,lino-ricardo1} which provided only a partial picture for three isotropy summands, by computing the stability of equilibiria but not the transient neither the limiting dynamics. On the other hand \cite{buzano, lino-ricardo2} considered the transient and limiting dynamics for two isotropy summands which, by the methods of the present article, reduces to one dimensional dynamics, where there are much fewer possibilities for the collapses to occur. 

The methods we introduce point to further generalizations for flag manifolds with four or more summands \cite{Arvanitoyeorgos-Chrysikos-2009, Arvanitoyeorgos-Chrysikos-Sakane-2013} and potentially offers insight to the classification problem of invariant Einstein metrics on compact homogeneous spaces  \cite{bohm2004variational}.


Here we consider the homogeneous Ricci flow on (generalized) flag manifolds of a compact connected simple Lie group $G$, whose isotropy representation splits into three irreducible components (isotropy summands). A flag manifold of $G$ is determined by a choice of subset $\Theta$ of simple roots, the ones with three isotropy summands were classified by Kimura \cite{kimura} into two classes: \textit{Type II} and of \textit{Type I}, according to the possible highest heights of the  chosen roots in $\Theta$ (see Tables \ref{tab-typeI} and \ref{tab-exep} for the list of these spaces). It is known that one of the main differences between these two classes is the number of invariant Einstein metrics: each flag manifold in the first class admits exactly four invariant Einstein metrics (up to scale) while those in the second class admits exactly three. Flag manifolds in both classes admit exactly one Einstein-K\"ahler metric.  We start our analysis with Type II since it includes the classical families of $SU(n)$ and $SO(2\ell)$ flag manifolds while Type I consists of finitely many flag manifolds related to exceptional Lie groups.

Next we summarize the main results of the paper. 
We suggest the reader to check the corresponding sections for details.


\begin{teoremaA}[Section \ref{sec:typeI}]
	The dynamics of the projected Ricci flow 
	of flag manifolds $\mathbb{F}$ with three isotropy summands of {\em Type II}
	are topologically equivalent.
	Their phase portraits and basins of attraction are described in Figure
	\ref{fig-introducao} (left).
	Moreover, the {\em Einstein non-K\"ahler} metric is a repeller and the three {\em K\"ahler-Einstein metrics} are hyperbolic saddles.
	
	In forward time an orbit converges to either an Einstein metric on the flag manifold or collapses the flag manifold to a point. In backward time an orbit converges to either an Einstein metric on the flag or collapses the flag manifold to a Riemannian symmetric space with a normal metric. The convergence is in the Gromov-Hausdorff sense.
	
%
\end{teoremaA}

\begin{figure}[ht]
\centering
\def\svgwidth{11cm}
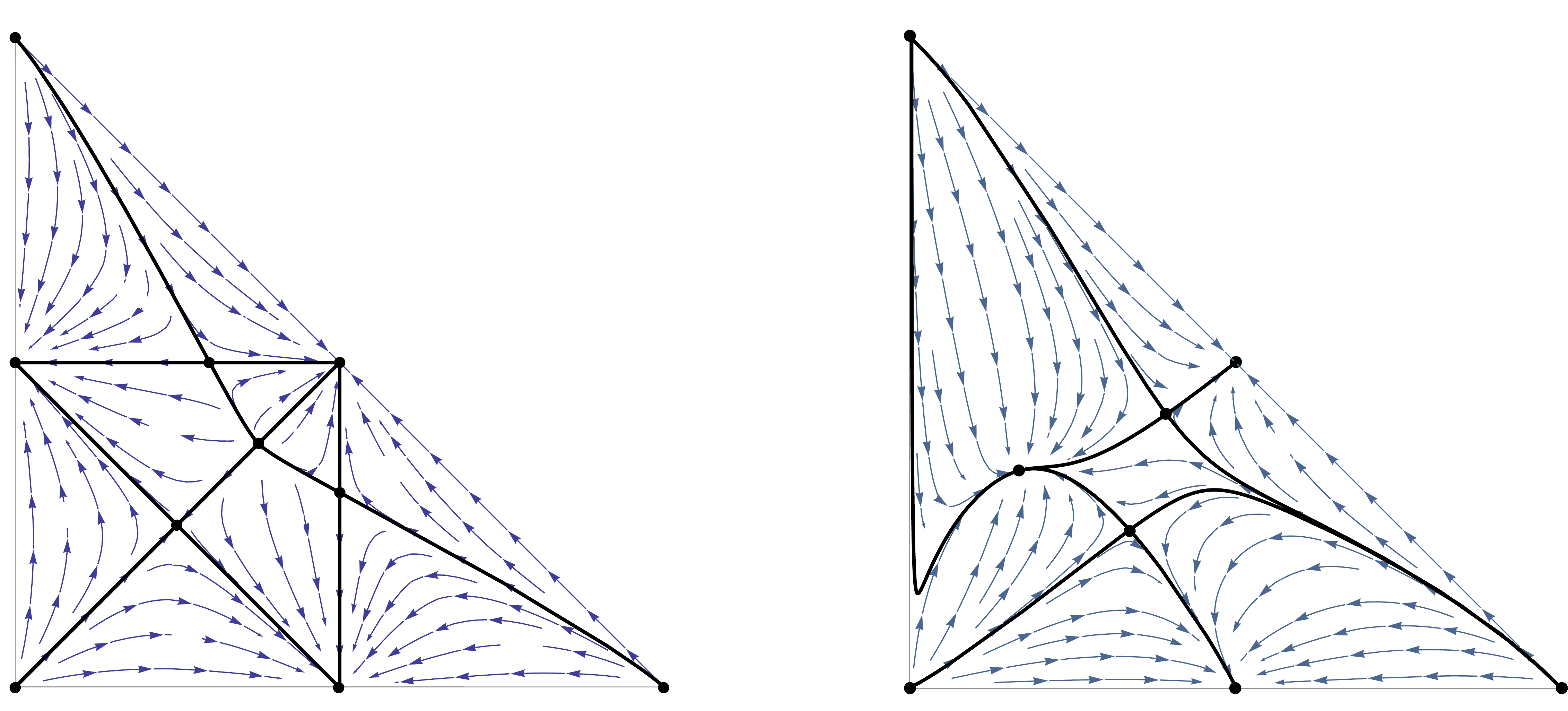
\caption{\label{fig-introducao}
Projected Ricci flow of Type II (left) and Type I (right).}
\end{figure}

\begin{teoremaA} [Section \ref{sec:typeII}] \label{teo2-intro}
	The dynamics of the projected Ricci flow 
	of flag manifolds $\mathbb{F}$ with three isotropy summands of {\em Type I}
	are topologically equivalent.
	Their phase portraits and basins of attraction are described in Figure \ref{fig-introducao} (right).
	Moreover, the {\em Einstein-K\"ahler metric} is a \info{alteracao}attractor and the two {\em Einstein non-K\"ahler metrics} are hyperbolic saddles. 
	
	In forward time an orbit converges to either an Einstein metric on the flag manifold or collapses the flag manifold to a point. In backward time an orbit converges to either an Einstein metric on the flag manifold or collapses the flag manifold to either \comentario{Llohann} a Riemannian {\em symmetric space} or a {\em Borel-de Siebenthal homogeneous space}, both with normal metric. The convergence is in the Gromov-Hausdorff sense.
	
\end{teoremaA}

It follows from Sesum \cite{sesum} and Petersen \cite{petersen} that the limit of any orbit in  the homogeneous Ricci flow is an Einstein manifold. Note that in this article this happens in three qualitatively different manners: convergence to an Einstein metric in the flag manifold, collapse to a Einstein metric on a homogeneous space with smaller but positive dimension, collapse to a point (which is trivially Einstein). 

Our general tools are a collapsing theorem (Section \ref{sec:gh}) and the projected Ricci flow (Section \ref{preliminaries}).  The main theorems above are then obtained {\em a posteriori} after exhausting the analysis for the families of flag manifolds with three isotropy summands of Type II and I (Sections \ref{sec:typeI} and \ref{sec:typeII}, respectively).  The symbolic and numerical calculations were carried out with the {\em Mathematica}\textsuperscript{TM} software package: software code used in the article is available upon request to the authors.

We start the paper recalling some preliminar results about Ricci flow, Gromov-Hausdorff convergence and flag manifolds.

\section*{Acknowledgments}

The authors would like to thank Pedro Solórzano for pointing out a first instance of Theorem \ref{thm:collapse}  involving holonomy, and  \info{incluido} also thank Jorge Lauret and Cynthia Will for their remarks and for pointing out a mistake in Theorem \ref{teo2-intro} in a previous version of this manuscript.  The last author is supported by  404266/2016-9 (CNPq); L.\ Grama is partially supported by  2018/13481-0 (FAPESP) and 305036/2019-0 (CNPq); R.\ M.\ Martins is partially supported by 434599/2018-2 (CNPq) and 2018/03338-6 (FAPESP); M. Patrão is partially supported by 88887.466805/2019-00 (CAPES-PRINT).

%
%
%
%
%
%

\section{Preliminaries}
\subsection{The Ricci flow of invariant metrics}
Let $M$ be a manifold. Then a family of Riemannian metrics $g(t)$ in $M$ is called a Ricci flow if it satisfies
\begin{equation}
\label{ricci-flow}
\frac{\partial g}{ \partial t}=-2\Ric(g).
\end{equation}
where we omit the parameter $t$ on $g$, whenever this is convenient.
One gets essentially the same geometry when $g$ is rescaled  by a constant $\lambda > 0$. Moreover, 
$\Ric(\lambda g) = \Ric(g)$. It follows that the Ricci operator $r(g)$, given by
\begin{equation}
\label{eq-ricciop}
\Ric(g)(X, Y) = g( r(g)X, Y)
\end{equation}
is homogeneous of degree $-1$: $r(\lambda g) = \lambda^{-1} r(g)$, and so is the scalar curvature $S(g) = \rm{tr}(r(g))$. 

One can \textit{gauge away} the scale $\lambda$ by normalizing the flow. 
For instance, if $M$ is compact and orientable one can normalize the flow to preserve volume as follows. In oriented local coordinates $x_1, \ldots, x_n$ at $p \in M$, the Riemannian volume form of a metric $g$ is given by
$
\vol(g)_p = \sqrt{\det(g_{ij})} \, dx_1 \wedge \ldots \wedge dx_n
$,
where $g_{ij} = g_p(\partial/\partial x_i, \partial/\partial x_j)$. 
Let $g(t)$ be 1-parameter family of Riemannian metrics and denote by prime the derivative with respect to $t$, then
$
\vol'(g(t))_p = \frac{1}{2} \tr( g_{ij}' (t) ) \vol(g)_p
$,
since we have the differential $d(\det(h))_h H = \det(h) \tr H$, where $h$, $X$ are square matrices, $h$ invertible.  Denoting by $A(t)$ the operator corresponding to $g'(t)$ under the metric $g(t)$, that is, $g'(t)(X,Y) = g(t)( A(t)X, Y )$, and choosing local coordinates at $p$ such that the $\partial/\partial x_i$ form an orthonormal frame at $p$, we get that
\begin{equation}
\label{eq-derivada-volume-pontual}
\vol'(g(t))_p = \frac{1}{2} \tr A(t) \vol(g)_p
\end{equation}
If $g(t)$ comes from the Ricci flow \eqref{ricci-flow}, we thus get that 
\begin{equation}
\label{eq-derivada-ricci-pontual}
\vol'(g(t))_p = - S(g(t))_p  \vol(g(t))_p, \qquad \vol'(M,g(t)) = - T(g(t)) \vol(M,g) 
\end{equation}
where $\vol(M,g) = \int_M \vol(g)$ is the total volume and $T(g) = \int_M S(g) \vol(g) / \vol(M,g)$ the total scalar curvature of the metric $g$. 
Let $d = \dim M$, {\em Hamilton's Ricci flow} on the space of unit volume metrics
is (see \cite{ziller})
\begin{equation}
\label{hamilton-ricci-flow}
\frac{\partial g}{ \partial t}=-2\left( \Ric(g) - \frac{T(g)}{ d }  g \right)
\end{equation}
which is a slight modification of the {\em gradient Ricci flow} of $T(g)$, given by
\begin{equation}
\label{gradient-ricci-flow}
\frac{\partial g}{ \partial t}=-2\left( \Ric(g) - \frac{S(g)}{ d }  g \right)
\end{equation}
Both flows preserve metrics of unit volume, by \eqref{eq-derivada-volume-pontual}. Its equilibria are the metrics satisfying $\Ric(g) = \lambda g$, for some $\lambda \in \R$, the so called {\em Einstein metrics}.

Now assume that $M=G/K$ is a compact homogeneous space, 
a $G$-invariant metric $g$ on $M$ is determined by its value $g_b$ at the origin $b=K$, which is a $\Ad_G(K)$-invariant inner product.  Just like $g$, the  Ricci tensor $\Ric (g)$ and the scalar curvature $S(g)$ are also $G$-invariant and completely determined by their values at $b$, $\Ric(g)_b = \Ric(g_b)$, $S(g)_b=S(g_b)$.  
Being a $G$-invariant scalar function on $M=G/K$, the scalar curvature $S(g)$ is constant, so that $T(g) = S(g_b) \vol(g)$. 
Taking this into account, the Ricci flow equation (\ref{ricci-flow}) becomes the autonomous ordinary differential equation in the vector space of bilinear forms on the tangent space $T_b M$, known as the {\em homogeneous Ricci flow}
\begin{equation}
\label{inv-ricci-flow}
\frac{dg_b}{dt}=-2\Ric(g_b)
\end{equation}
Also, the corresponding Hamilton \eqref{hamilton-ricci-flow} and gradient \eqref{gradient-ricci-flow} Ricci flow on invariant metrics of unit volume coincide, becoming the {\em homogeneous gradient Ricci flow}
\begin{equation}
\label{normaliz-ricci-flow}
\frac{dg_b}{dt}=-2\left( \Ric(g_b) - \frac{S(g_b) }{ d } g_b \right)
\end{equation}
which is the gradient flow of the functional $T(g) = S(g_b)$ (see \cite{bohm2004variational}).

As in the unit-volume normalization \eqref{normaliz-ricci-flow}, one can normalize \eqref{inv-ricci-flow} by choosing an hypersurface in the (finite dimensional) space of invariant metrics which is transversal to the semi-lines $\lambda\mapsto \lambda g_b$. In the aforementioned case, the hypersurface consists on unit volume metrics and is unbounded. In order to study the limiting behavior of the homogeneous Ricci flow, in Section \ref{preliminaries} we will normalize it instead to a simplex and rescale it to get a polynomial vector field.

Next we introduce some notation and conventions.  Let the trivial coset $b=K$ be the basepoint of $G/K$, then the map $\g \to T_b(G/K)$ that assigns to $X \in \g$ the induced tangent vector $X \cdot b = d/dt (\exp(tX)b) |_{t=0}$ is surjective with kernel the isotropy subalgebra $\k$. Using that $g \in G$ acts in tangent vectors by its differential, we have that
\begin{equation}
\label{eq-induzido}
g ( X \cdot b ) = ( \Ad(g)X ) \cdot g b
\end{equation}
In what follows we assume that the homogeneous space $M=G/K$ is \textit{reductive}, with reductive decomposition $\mathfrak{g}=\mathfrak{k}\oplus\mathfrak{m}$ (that is, $[\mathfrak{k},\mathfrak{m}]\subset \mathfrak{m}$). Then $\m$ is $\Ad_G(K)$-invariant so that, by equation (\ref{eq-induzido}), the restriction $\m \to T_b(G/K)$ of the above map is a linear isomorphism that intertwines the isotropy representation of $K$ in $T_b(G/K)$ with the adjoint representation of $G$ restricted to $K$ in $\m$.  This allows us to identify $T_b(G/K) = \m$ and the $K$-isotropy representation with the $\Ad_G(K)$-representation.

We further assume that $G$ is a  compact connected simple Lie group and that the isotropy representation of $G/K$ decomposes $\mathfrak{m}$ as 
\begin{equation}\label{deco-iso}
\mathfrak{m}=\mathfrak{m}_1\oplus \ldots \oplus \mathfrak{m}_n
\end{equation}
where $\mathfrak{m}_1,...,\mathfrak{m}_n$ are  irreducible pairwise non-equivalent isotropy representations.
A source of examples satisfying the assumptions above are {\em generalized flag manifolds} (see Section \ref{sec:flag} for details). With the assumptions above,  all invariant metrics are given by
\begin{align}
\label{eq-compon-metr}
g_b&=x_1B_1+\ldots + x_nB_n
\end{align}
where $x_i>0$ and $B_i$ is the restriction of the (negative of the) Cartan-Killing form of $\mathfrak{g}$ to $\mathfrak{m}_i$. We also have 
\begin{align}
\label{eq-compon-ricci}
\Ric (g_b)&=y_1 B_1 + \ldots + y_nB_n
\end{align}
where 
$y_i$ is a function of $x_1, \ldots, x_n$.
Therefore, the Ricci flow (\ref{inv-ricci-flow}) becomes the autonomous system of ordinary differential equations
\begin{equation}
\label{eq-ricci-flow-coords}
\frac{dx_k}{dt}= -2 y_k, \qquad \ k=1,\ldots , n
\end{equation}
%
Next, we write the Ricci flow equation in terms of the {Ricci operator} $r(g)_b$.  
Since $r(g)_b$ is invariant under the isotropy representation,
$r(g)_b|_{\mathfrak{m}_k}$ is a multiple $r_k$ of the identity.
From \eqref{eq-ricciop}, \eqref{eq-compon-metr} and  \eqref{eq-compon-ricci}, we get
$$
y_k = x_k r_k
$$
and equation \eqref{eq-ricci-flow-coords} becomes
\begin{equation}
\label{eq-ricci-flow-final}
\frac{dx_k}{dt}= -2 x_kr_k
\end{equation}
We denote by $R(x_1, \ldots, x_n)$ the vector field on the right hand side of \eqref{eq-ricci-flow-final}, with phase space $\mathbb{R}_+^n = \{ (x_1, \ldots, x_n) \in  \mathbb{R}: \, x_i > 0 \}$. Moreover, $x \in \mathbb{R}_+^n$ corresponds to an  Einstein  if and only if $R(x) = \lambda x$, for some $\lambda > 0$.
The homogeneous gradient Ricci flow (\ref{normaliz-ricci-flow}) on invariant metrics then becomes
\begin{equation}
\label{eq-ricci-flow-normaliz-final}
\frac{dx_k}{dt}= -2 x_kr_k - \frac{2}{d} S(x) x_k
\end{equation}
where 
\begin{equation}
\label{eq-scalar-curvature}
S(x) =  \sum_{i=1}^{n} d_i r_i
\end{equation}
is the scalar curvature, $d_k = \dim \m_k$.

\subsection{Flag manifolds}\label{sec:flag}

For the sake of completeness, we recall some results and notations about compact Lie groups and its flag manifolds (see \cite{Arvanitoyeorgos-Chrysikos-Sakane-2013, helgason, itoh} for details and proofs).  Let a compact connected Lie group $G$ have Lie algebra $\g$ and a maximal torus $T$ with Lie algebra $\t$.  We have that $\g$ is the compact real form of the complex reductive Lie algebra $\g_\C$.  The adjoint representation of the Cartan subalgebra $\h = \t_\C$ splits as the root space decomposition
$
\g_\C = \h \oplus \sum_{\alpha \in \Pi} \g_\alpha
$
with root space
$$
\g_\alpha = \{ X \in \g_\C :\, \ad(H) X =  \alpha(H) X, \, \forall H \in \h  \},
$$
where $\Pi \subset \h^*$ is the root system.  
Consider
$$
\m_\alpha = \g \cap ( \g_\alpha \oplus \g_{-\alpha} )
$$
and let $\Pi^+$ be a choice of positive roots, then $\g$ splits as
$$
\g = \t \oplus \sum_{\alpha \in \Pi^+} \m_\alpha.
$$
Denote by $\Sigma$ the subset of simple roots corresponding to $\Pi^+$.

A flag manifold of $G$ is a homogeneous space $G/K$ where $K$ is the centralizer of a torus. We have that $K$ is connected and w.l.o.g.\ we may assume that $T \subset K$.  Recall that $T$ is the centralizer of $\t$. More generally, one can take $K = G_\T$, where the latter is the centralizer of
$$
\t_\Theta = \{ H \in \t:~ \alpha(H) = 0, \, \alpha \in \Theta \}
$$
and $\Theta$ is a subset of the simple roots $\Sigma$ which, in rough terms, furnishes the block structure of the isotropy $G_\T$. The Lie algebra $\k = \g_\T$ splits as
$$
\k = \t \oplus \sum_{\alpha \in \langle \T \rangle^+} \m_\alpha,
$$
where $\langle \T \rangle^+$ is the set of positive roots given by sums of roots in $\T$.  We denote
\begin{equation}
\label{eq-def-flag}
\F_\T = G/G_\T
\end{equation}
with basepoint $b=G_{\T}$.  Since the center $Z$ of $G$ is contained in $T$, $Z$ contained in $G_\T$. Taking the quotient of both $G$ and $G_\T$ by $Z$ in (\ref{eq-def-flag}), we obtain the same flag manifold. Note that $G/Z$ is isomorphic to the adjoint group of $\g$.  Thus, $\F_\T$ depends only on the Lie algebra $\g$ of $G$, which we can assume to be simple.

A $G_\T$-invariant isotropy complement of $\F_\T$ is given by
$$
\m \, = \sum_{\alpha \in \Pi^+ - \langle \T \rangle^+} \m_\alpha,
$$
so that $\F_\T$, with $\g = \k \oplus \m$, is reductive and the isotropy representation of $\F_\T$ is equivalent to the adjoint representation of $G_\T$ in $\m$. This representation is completely reducible and can be uniquely decomposed as the sum of non-equivalent irreducible representations
$$
\m = \m_1 \oplus \cdots \oplus \m_n,
$$
where each $\m_k$ is an appropriate sum of $\m_\alpha$'s (see \cite{siebenthal}).

\subsection{Flag manifolds with three isotropy summands}

According to \cite{kimura}, there exist two classes of flag manifolds with three isotropy summands, of \textit{Type II} and of \textit{Type I}, depending on the Dynkin mark of the roots in $\Pi^+\setminus \Theta ^+$.
Recall that the Dynkin mark of a simple root $\alpha\in\Sigma$ is the coefficient $\mrk(\alpha)$ of $\alpha$, in the expression of the highest root as a combination of simple roots.
Let the decomposition into irreducible components of $\m$ be
$$
\mathfrak{m}=\mathfrak{m}_1 \oplus \mathfrak{m}_2 \oplus  \mathfrak{m}_3
$$
and recall that $d_i$ is real dimension of the corresponding isotropy component $\mathfrak{m}_i$, $i=1,2,3$.  

\begin{teorema}[\cite{kimura}]
	We have that
	\begin{enumerate}[i)]
		\item 
		The generalized flag manifold $G/G_\Theta$ has three isotropy summands if, and only if, the set $\Theta\subset \Sigma$ is given by
		$$
		\begin{array}{l | l}
		\mbox{Type } & \\
		\hline
		II & \Sigma\setminus \Theta = \{ \alpha, \, \beta:~ \mrk (\alpha) = \mrk (\beta) = 1  \} 
		\\  \hline
		I &  \Sigma\setminus \Theta = \{ \alpha: \mrk (\alpha) = 3  \}
		\end{array}
		$$
		
		\item The \textit{Type I} flag manifolds are listed in Table  \ref{tab-exep}. Each one admits exactly three invariant Einstein metrics (up to scale); exactly one of them is Einstein-K\"ahler.
		
		\item The \textit{Type II} flag manifolds are listed in Table  \ref{tab-typeI}.  Each one admits exactly four invariant Einstein metrics (up to scale); exactly one of them is Einstein-K\"ahler.		
	\end{enumerate}
\end{teorema}

\begin{table}[h!]
	\caption{ \label{tab-exep} Type I flag manifolds with three isotropy summands}
	\begin{center}
		\begin{tabular}{lllllll}
			Flag Manifold   & & $d_1$ & & $d_2$ & & $d_3$ \\ \hline \hline
			$E_8/E_6 \times SU(2) \times U(1)$ & & 108 & & 54 & & 4 \\ \hline
			$E_8/SU(8) \times U(1)$ && 112 && 56 && 16 \\ \hline
			$E_7/SU(5)\times SU(3) \times U(1)$  && 60 && 30 && 10 \\ \hline
			$E_7/SU(6)\times SU(2) \times U(1)$  && 60 && 30 && 4 \\ \hline
			$E_6/SU(3)\times SU(3) \times SU(2) \times U(1)$ && 36 && 18 && 4 \\ \hline
			$F_4/SU(3)\times SU(2) \times U(1) $ && 24 && 12 && 4 \\ \hline
			$G_2/U(2)$ && 4 && 2 && 4
		\end{tabular}
	\end{center}
\end{table}

\begin{table}[h!]
	\caption{ \label{tab-typeI} Type II flag manifolds with three isotropy summands}
	\begin{center}
		\begin{tabular}{lllllll}
			Flag Manifold   & & $d_1$ & & $d_2$ & & $d_3$ \\ \hline \hline
			$SU(m+n+p)/S(U(m)\times U(n) \times U(p))$ && $2mn$ &&$2mp$ && $2np$ \\ \hline
			$SO(2\ell)/U(1)\times U(\ell-1)$, $\ell \geq 4$ && $2(\ell-1)$ && $2(\ell -1)$ && $(\ell-1)(\ell-2)$ \\ \hline
			$E_6/SO(8)\times U(1)\times U(1)$ && 16 && 16 && 16
		\end{tabular}
	\end{center}
\end{table}

\section{A collapsing result \label{sec:gh}}

Buzano \cite{buzano} proves that a sequence of $G$-invariant Riemannian submersion $G/K\to G/H$ with shrinking fibers converges in the Gromov-Hasudorff sense to $G/H$.  In this article we generalize it in two important aspects, which are the main difficulties of our proof:  the shrinking directions need not define  a integrable distribution, therefore one must resort to Chow-Rascheviski Theorem to correctly identify the shrinking sets; furthermore, although the shrinking sets can be seen as the fibers of a submersion $G/K\to G/H$, the quotient map is not usually compatible with a Riemannian structure, forcing us into length-spaces techniques (which allows us to non-continuously divide  curves into  little pieces), naturally leading us to a Finsler metric. Indeed, considering the endeavour in metric geometry, our method and results takes a direction closer to that in Solórzano \cite{solorzano}. Our generalization is essential for the results of this article since in many of the collapses that appear here, the shrinking direction is non-integrable.

Let us start by quickly recalling  a definition of \textit{Gromov-Hausdorff distance} and its induced topology (see \cite{burago,gromov} for details).  A {\em correspondence} between the metric spaces $(A,d_A)$ and $(B,d_B)$ is a subset  $S\subseteq A \times B$ such that both projections  $S\to A$ and $S\to B$ are onto. If, in addition, $|d_A(p_1,q_1)-d_B(p_2,q_2)|<\epsilon$, for every $(p_1,p_2)$, $(q_1,q_2) \in S$, then we denote $A\sim_\epsilon B$.

\begin{defi}
	The Gromov--Hausdorff distance between $(A,d_A)$ and $(B,d_B)$ is defined by
	$$
	d_{GH}(A,B)=\inf \{ \epsilon \geq 0 ~:~ A\sim_\epsilon B \}.
	$$
	If there is no $\epsilon$ such that $A\sim_\epsilon B$, we write $d_{GH}(A,B)=\infty$. 	
	
	We say that a family of metric spaces $\{(X_t,d_t)  \}_t$, $t \in \R$, converges to $(X,d)$ in the Gromov-Hausdorff sense and write 
	$$
	\lim_{t \to \infty} (X_t,d_t) = (X,d)
	$$
	when $d_{GH}(X_t,X) \to 0$ as $t \to \infty$. 
\end{defi}

Let us fix some choices and notations.
Fix in $\g$ a $G$-invariant inner product $B$ and, given a subspace $\lie v \subseteq \g$, denote by $\lie v^\perp$ its $B$-orthogonal complement in $\g$. Identify $T_b(G/K)$ with $\lie m=\lie k^\perp$, which is $K$-invariant.
A $G$-invariant bilinear form $\beta$  in $G/K$ is defined by is value at the basepoint $b$, which is a $K$-invariant inner product on $\m$ that we will also denote by $\beta$, and vice-versa. Thus we can speak of convergence of $G$-invariant bilinear forms of $G/K$ by using the natural topology of bilinear forms on $\m$. If $\beta$ is a $G$-invariant Riemannian metric on $G/K$ or, equivalently a $K$-invariant inner product on $\m$, we denote its induced curve length by
\begin{equation}
\label{comprim-riemann}
\ell_\beta(c)=\int_0^1 | \dot c (t) |_\beta\, dt,
\end{equation}
where $ | \cdot |_\beta$ denotes the norm associated to $\beta$, and the corresponding Riemannian distance on $G/K$ by
$$
d_\beta(pK, qK) = \inf \{ \ell_\beta(c) :~ c \in C^1([0,1],G/K),~
c(0) = pK,~ c(1) = qK \} .
$$
Denote the restriction $\beta|_{\m \times \m}$ by $\beta|_\m$.
With these notations we have 
$\ell_\beta = \ell_{\beta|_\m}$ and 
$d_\beta = d_{\beta|_\m}$.

Now suppose that $g_t$ is a family of $G$-invariant metrics of $G/K$, $t > 0$, which converges to the bilinear form $g$ when $t \to \infty$.  Then $g$ is determined by a non-negative $K$-invariant bilinear form on $\m$ which we also denote by $g$.
Consider
\begin{equation*}
\m_0=\ker g=\{X\in \mathfrak m~:~g(X,\mathfrak m)=0\}
\end{equation*}
which is $K$-invariant, since $g$ and $\m$ are.
Let
$$
\h =  \text{ Lie algebra generated by }\m_0 \oplus \k
$$
which is $K$-invariant, since $\m_0$ and $\k$ are. Take $H < G$ as the connected Lie subgroup with Lie algebra $\h$. Since $\m_0$ is $\k$-invariant, it follows that the subalgebra generated by $\m_0$ is $\k$-invariant so that 
$\h$ coincides with the sum of $\k$ with the Lie algebra generated by $\m_0$, which guarantees that the distribution induced by $\m_0$ is \textit{bracket generating} in $H/K$ (see section \ref{sec:GH1} for details). Suppose that $H$ is closed and identify $T_b(G/H)$ with
$$
\n = \h^\perp \subset \k^\perp = \m
$$
which is $K$-invariant. Note that $g|_\n$ is a $K$-invariant inner product. 
It follows that, as $t\to \infty$, the fibers of the natural projection $\pi:G/K\to G/H$ collapse. More precisely, we have the following:

\begin{teorema}\label{thm:collapse}
	Let $(G/K,g_t)$, $g$ and $H$ be as above. Then
	\begin{equation*}
	\lim_{t\to \infty}(G/K, d_{g_t})=(G/H,d_F),
	\end{equation*}
	where $d_F$ is the distance metric induced by the (not necessarily smooth) Finsler norm $F:T(G/H)\to \bb R$,
	$$
	F(X)=\inf\left\{|Y|_g~|~ Y\in T(G/K),~d\pi(Y)=X  \right\}.
	$$
\end{teorema} 
The norm $F$ can be interpreted as the shortest direction one could leave the coset $H/K$ to cosets in the direction of $X$. In this sense, it is reasonable to conceive $(G/H,d_F)$ as the limiting space, since the diameter of $H$-cosets goes to zero, so one can freely moves inside each coset and choose the point with the shortest exit. 
Note that, for $X\in T_{pH}(G/H)$, the infimum in $F(X)$ is computed among vectors along the whole fiber $pH$, not only in $T_{pK}(G/K)$.

Since $F$ is clearly $G$-invariant, it depends only on its value at $F|_\n$, given by the following Lie-algebraic  description
\begin{align}\label{eq:Finsler-Lie}
F:  \n&\to \mathbb R\\
\nonumber	X &\mapsto \min\{ |\Ad(h)X + Z|_{g}:~ h\in H,~ Z \in \m \cap \h \}.
\end{align}
To verify the equality between both definitions of $F$, since $F$ is $G$-invariant, it is sufficient to show that  $d\pi(Y)=X$ if and only if $Y=h( \Ad(h^{-1})X+Z )$ for some $h\in H$ and $Z\in \lie m\cap \lie h$. To this aim, first note that $\ker d\pi|_\m = \m \cap \h$, so that $d\pi|_\m$ is the projection of the direct sum $\m = \n \oplus (\m \cap \h)$ onto $\n$. Also note that $d\pi(Y)=X \in \n$, implies $Y\in  T_{hK}(G/K)=h \m$ for some $h\in H$. Thus we can decompose $h^{-1} Y = W + Z$, for $W \in \n$, $Z \in \h \cap \m$.  By the equivariance of $\pi$ (and since $H$ acts in $\n$ by the adjoint action), it follows that
$$
W = d\pi( h^{-1} Y ) = h^{-1} d\pi(Y) = \Ad(h^{-1}) X.
$$
Therefore,  $Y = h( \Ad(h^{-1}) X + Z)$.  Following along the same lines, given $h\in H$ and $Z\in \h\cap\m$, we have
\[d\pi(h (\Ad(h^{-1}) X + Z) ) = h d\pi( \Ad(h^{-1}) X + Z ) = h ( \Ad(h^{-1}) X ) = X.
\]
%
%
%
%

Equation (\ref{eq:Finsler-Lie}) immediately implies the following.

\begin{corolario}
Suppose further that   $g(\n,\m \cap \h)=0$ and that $g|_\n$ is $\Ad_G(H)$-invariant. Then the Finsler norm $F$ is induced by the Riemannian metric $g|_\n$.
\end{corolario}


For the proof of Theorem \ref{thm:collapse} we use sub-Riemannian techniques, in contrast to \cite[Proposition 2.6]{buzano}, which uses Riemannian submersions. 
The latter result is recovered when $\mathfrak m_0\oplus \mathfrak k$ is a subalgebra. 
%
%
%
The proof is divided in two parts, \ref{sec:GH1} and \ref{sec:GH2}, and we now fix notation.  Consider the natural projection $\pi:G/K\to G/H$. 
For $p \in G$ we denote by $pH$ both the corresponding point in $G/H$ and the corresponding coset in $G/K$, and it will be clear from the context which one is considered.  Consider in $G/H$ the distance
$$
\tilde d_{g_t}(pH,qH)=d_{g_t}(pH,qH)
$$
given by the $g_t$-distance of the corresponding fibers in $G/K$, which is not necessarily induced by a Riemannian metric in $G/H$.  
Recall that $\n=\lie h^\perp$ is the $B$-orthogonal complement of $\lie h$ and note that $g_t|_\n$ does not necessarily induces an invariant Riemannian metric in $G/H$ since is not necessarily $\Ad_G(H)$-invariant.
We  first show that the families $(G/K,d_{g_t})$ and  $(G/H,\tilde d_{g_t})$ have the same limit (Corollary \ref{cor:collapse_part1})  and then characterize this limit  (Lemma \ref{lemma:collapse_part2}). 

\subsection{Proving that
$\lim_{t\to \infty} (G/K,d_{g_t}) =  \lim_{t\to \infty} (G/H,\tilde d_{g_t})$}
\label{sec:GH1}
To this aim, we consider the simplest case of a correspondence:
$$
S_t = \{ (pK,  \, pH ):\,  p \in G \}  \subseteq (G/K,d_{g_t}) \times (G/H, \tilde d_{g_t}). 
$$ 
Clearly $S_t$ projects surjectively over both factors.  It is only left to prove that, given $\epsilon > 0$, there exists $T$ such that
\begin{equation}\label{eq:uniform.convergence}
| d_{g_t}( pK, qK ) -   \tilde d_{g_t}( pH, qH ) | = | d_{g_t}( pK, qK ) -   d_{g_t}( pH, qH) |
< \epsilon
\end{equation}
for all $p, q \in G$ and $t>T$.   Moreover, since $d_{g_t}(pK,qK)\geq d_{g_t}(pH,qH)$, it is sufficient to show that $d_{g_t}( pK, qK ) -   d_{g_t}( pH, qH )<\epsilon$. 

To estimate $d_{g_t}(pK,qK)$ we consider the concatenation $c=c_3c_2c_1$, where: $c_2$ is a minimizing curve connecting  the fibers $pH$ to $qH$, thus realizing the fiber distance $d_{g_t}(pH,qH)$, which exists since $H$ is closed in $G$, hence compact; $c_1$ is in the fiber $pH$ and connects $pK$ to $c_2(0)$; $c_3$ is in the fiber $qH$ and connects $c_2(1)$ to $qK$. We get
\begin{equation}\label{eq:estimative1}
d_{g_t}(pK,qK)\leq  d_{g_t}(pH,qH)+\ell_{g_t}(c_1)+\ell_{g_t}(c_3),
\end{equation}
where $\ell_{g_t}(c_j)$ stands for the length of $c_j$  in the metric $g_t$.
Equation \eqref{eq:uniform.convergence} follows from \eqref{eq:estimative1} once we prove that we can uniformly bound the lengths $\ell_{g_t}(c_1),\ell_{g_t}(c_3)$ of curves in the fibers
by some family of constants $C_t$, whose limit is zero. 
Since $g_t|_{\mathfrak m_0}\to 0$, the situation naturally leads us to sub-Riemannian geometry, through  Chow-Raschevskii Theorem, which we recall below.

Let $M$ be a compact connected smooth manifold and $\mathcal H\subseteq TM$ a {\em bracket generating distribution}, i.e., $TM$ is generated by vectors of the form $[X_1,[X_2,[...,$ $[X_{j-1},X_j]...]]]$, where the $X_i$ are local sections of $\mathcal H$.  
A \textit{horizontal curve} is a curve in $M$ which is tangent to $\mathcal H$.
If $\beta$ is a Riemannian metric for the distribution $\mathcal H$, we define the $\beta$-length of an horizontal curve $c$ by
\begin{equation}\label{eq:comprimento}
\ell_{\beta,\mathcal H}(c)=\int_0^1 | \dot c (t) |_\beta dt,
\end{equation}
where $| \cdot |_\beta$ is the norm associated to $\beta$.  Chow-Raschevskii Theorem guarantees that $\ell_{\beta,\mathcal H}$  indeed defines a metric

\begin{teorema}[\cite{agrachev}, Theorem 3.31]
	\label{thm:ChowRaschevskii}
	Let $M,\mathcal H, \beta$ be as above. Then 
	\begin{equation*}
	d_{\beta,\mathcal H}(p,q)=\inf\{\ell_{\beta,\mathcal H}(c)~:~\dot c\in \mathcal H,~c(0)=p,~c(1)=q \}
	\end{equation*}
	defines a metric on $M$. Moreover,
	\begin{enumerate}
		\item The topology induced by $d_{\beta,\mathcal H}$ is the topology of $M$,
		\item The $d_\beta$-diameter of $M$, $\rm{diam}_{\beta,\mathcal H}(M)$, is finite,
		\item Between every pair $p,q\in M$, there is a curve $c$, $\dot c\in \mathcal H$, with $\ell_{\beta,\mathcal H}(c)=d_{\beta,\mathcal H}(p,q)$.
	\end{enumerate}
\end{teorema}

In our context, $\mathfrak m_0$ defines  an invariant distribution $\mathcal H_0$ in $G/K$, given by $(\mathcal H_0)_{pK}=p\mathfrak m_0$. Which is, by the choice of $H$, bracket generating inside each fiber $pH$ of the projection $\pi:G/K \to G/H$
(see \cite[Lemma 3.32]{agrachev}). 
Thus, item 3 of Chow-Raschevskii's Theorem applied to the fiber $H$ of $\pi$ implies that 
$$
\rm{diam}_{g_t}(H) \leq \rm{diam}_{g_t,\mathcal H_0}(H).
$$
The $G$-invariance of $g_t$ gives $\rm{diam}_{g_t,\mathcal H_0}(pH)=\rm{diam}_{g_t,\mathcal H_0}(H)$. Since we can always choose  $\ell_{g_t}(c_1),\ell_{g_t}(c_3)\leq \rm{diam}_{g_t}(H)$ in \eqref{eq:estimative1}, to complete the first part of the proof of Theorem \ref{thm:collapse}, it is then sufficient to show the following

\begin{lema}\label{lem:SRdiameter}
	$\lim\limits_{t \to\infty} \rm{diam}_{g_t,\mathcal H_0}(H)=0$.
\end{lema}
\begin{proof}
	Recall the fixed $G$-invariant inner product $B$ of $\g$.
	Since $g_t:\mathfrak m_0\times \mathfrak m_0\to \mathbb R$ is a sequence of inner products converging to zero, there is a sequence $C_t>0$, $C_t\to 0$, such that
	\begin{equation}
	\label{eq:SRdiameter}
	g_t(X,X)\leq C_t^2 B(X,X)
	\end{equation}
	for all $X\in \mathfrak m_0$. By considering the respective $G$-invariant metrics, equation (\ref{eq:SRdiameter}) also holds for all $X$ tangent to $\mathcal H_0$.	
	Thus, for a $\cal H_0$-horizontal curve $c$, we have that
	$\ell_{g_t}(c) \leq C_t \, \ell_{B|_\m}(c)$. Thus,
	$$
	\rm{diam}_{g_t,\mathcal H_0}(H) \leq C_t \, \rm{diam}_{B|_\m,\mathcal H_0} (H).
	$$
	The Lemma follows then from item 2 of Chow-Raschevskii's Theorem which implies that $\rm{diam}_{B|_\m,\mathcal H_0}(H)$ is finite.
\end{proof}

\begin{corolario}\label{cor:collapse_part1}
$\lim_{t\to \infty} (G/K,d_{g_t}) =  \lim_{t\to \infty} (G/H,\tilde d_{g_t})$
\end{corolario}

\begin{remark}
The argument so far can be carried out in a much more general situation: instead of a coset foliation, one could consider the orbits $\cal O$ of a a family of vector fields (as in \cite{sussmann}), provided that the induced diameters $\rm{diam}_{\beta_t,\cal H}(\cal O)$ uniformly goes to zero. 

One interesting instance where it happens is the case of a family of Riemannian submersions $(M,g_t)\to (B,\bar g_t)$  with a shrinking base (compare Solórzano \cite[Theorem 3.8]{solorzano}). In particular, if the horizontal space is bracket generating (or, more generally, if the submersions has only one dual leaf) then the the total space $(M,\bar g_t)$ converges to a point.
\end{remark}

\subsection{Characterizing $\lim\limits_{t \to\infty} (G/H,\tilde d_{g_t})$}
\label{sec:GH2}
The remaining of the proof consists in studying the geodesics of the limit of $\tilde d_t$: we first prove that $\tilde d_{g_t}$ uniformly converges to an analogously defined $\tilde d_g$, then we observe that its geodesics are Lipschitz with respect to a normal homogeneous metric in $G/H$, concluding that they are $C^1$ outside a measure zero set. The proof is concluded by computing the dilatation of smooth curves.

Even though the limit $g$ is not a Riemannian metric in $G/K$, it still makes sense to speak of the $g$-length 
\[\ell_g(c)=\int_0^1|\dot c(\xi)|_gd\xi.\]
Define $\tilde d_g:G/H\times G/H\to \mathbb R$ accordingly
\begin{equation*}\label{eq:dg}
\tilde d_g(pH,qH)=\inf\left\{\ell_g(\tilde c):\tilde c \in C^1([0,1],G/K),~\pi\circ \tilde c=c,~
~ \tilde c(0)\in pH,~\tilde c(1)\in qH\right\}
\end{equation*}
For fixed $p,q \in G$, we clearly have that $\tilde d_{g_t}(pH,qH) \to \tilde d_g(pH,qH)$, so that $\tilde d_{g_t}$ pointwise converges to $\tilde d_g$. Next, we use Arzelà-Ascoli Theorem to show that this convergence is uniform. 

Consider the induced metric spaces 
$(G/K,d_{g_t})$, $(G/K,d_{B|_\m})$, $(G/H,d_{B|_\n})$, $(G/H,\tilde d_{g_t})$ and the natural map $\pi:G/K\to G/H$.
Since $g_t|_\m$ is convergent (therefore bounded) and $g|_\n$ is non-degenerate, there exist constants $c,C>0$ such that, for all $t>0$,
\begin{eqnarray*}
g_t(X,X) & \leq & C^2 B(X,X) \quad \hfill \forall X \in \mathfrak m
\\
\min\{g_t(\bar Y,\bar Y)\, :~\bar Y \in T(G/K),~d\pi(\bar Y)=Y\} & \geq & \,\,\,\, c^2 B(Y,Y) \quad  \forall Y\in  \n
\end{eqnarray*}
Let $\gamma:[0,1]\to G/K$ be a $d_{B|_\m}$-minimizing geodesic between $pK$ and $qK$. We have
$$
d_{B|_\m}(pK,qK) =\int_0^1 |\dot \gamma(\xi)|_{B|_\m}d\xi \geq C^{-1}\int_0^1 |\dot \gamma(\xi)|_{g_t}d\xi \geq C^{-1}d_{g_t}(pK,qK).
$$
On the other hand, let $\phi: [0,1]\to G/K$ be a $d_{g_t}$-minimizing geodesic between the cosets $pH$ and $qH$. Then, $\pi \circ \phi$ is a curve between $pH$ and $qH$ in $G/H$. Moreover
$$
\tilde d_{g_t}(pH,qH) = \int_0^1 |\dot \phi(\xi)|_{g_t}d\xi 
\geq 
c \int_0^1 |d\pi(\dot{\phi}(\xi))|_{B|_\n}d\xi 
\geq
cd_{B|_\n}(pH, qH)
$$
We conclude

\begin{lema}\label{lem:Lipschitz}
	There are constants $c,C>0$ such that, for all $t > 0$ and $p,q  \in G$,
	\begin{equation*}
	c\, d_{B|_{\n}}(pH,qH)\leq \tilde d_{g_t}(pH,qH)\leq d_{g_t}(pK,qK)\leq 
	C \, d_{B|_{\m}}(pK,qK)
	\end{equation*}
	In particular 
	\begin{enumerate}[$(i)$]
		\item the sequences $\tilde d_{g_t},d_{g_t}$ are uniformly equicontinuous and uniformly bounded;
		\item If a map  $f: (X,d)\to (G/H,\tilde d_{g_t})$, from a metric space $(X,d)$, is Lipschitz, so it is $f:(X,d)\to (G/H,d_{B|_\n})$.
	\end{enumerate}
	In particular, $(G/H,\tilde d_{g_t})$  converges to $(G/H,\tilde d_g)$ in the Gromov-Hausdorff sense and $(G/H,\tilde d_g)$ is a length space.
\end{lema}
\begin{proof}
Observe that, as family of functions, $d_{g_t}:(G/K\times G/K,d_{B|_\m}\times d_{B|_\m})\to \bb R$ and $\tilde d_{g_t}:(G/H\times G/H,d_{B|_\n}\times d_{B|_\n})\to \bb R$ are Lipschitz. The first one since
 \begin{multline*}
 d_{g_t}(pK,qK)-d_{g_t}(p'K,q'K)\\=d_{g_t}(pK,qK)-d_{g_t}(pK,q'K)+d_{g_t}(pK,q'K)-d_{g_t}(p'K,q'K)\\\leq d_{g_t}(pK,p'K)+d_{g_t}(qK,q'K)\leq C(d_{B|_\lie m}(pK,p'K)+d_{B|_\lie m}(qK,q'K) );
 \end{multline*}
the last one since $d_{B|_\m}(pH,qH)=d_{B|_\n}(pH,qH)$, thus $\tilde d_{g_t}(pH,qH)\leq Cd_{B|_\n}(pH,qH)$. Therefore both families $\tilde d_{g_t},d_{g_t}$ are Lipschitz,  with fixed Lipschitz constant $C$, concluding item $(i)$. Item $(ii)$ follows from the definition of a Lipschitz map. The last assertions follow from  $(i)$ and the pointwise convergence $\tilde d_{g_t}\to \tilde d_g$: Corollary \ref{cor:collapse_part1} and \cite[Example 7.4.4]{burago} guarantees that  $(G/H,\tilde d_g)$  is the Gromov-Hausdorff limit  of the length spaces $(G/K,d_{g_t})$, therefore it is a length space itself (\cite[Theorem 7.5.1]{burago} or \cite[Proposition 3.8]{gromov}).	
\end{proof}

\begin{remark}
	The knowledge \textit{a priori} that $(G/H,\tilde d_g)$ is a length space is key to our proof. Therefore, it is worth remarking that although $(G/H,\tilde d_{g_t})$ might not be length spaces, the degeneration of $g$ guarantees that $(G/H,\tilde d_g)$ is: compare $\tilde d_g$ with the \textit{quotient semi-metric} in \cite[Definition 3.1.12]{burago}.
\end{remark}

Gathering the information so far, we conclude that $(G/H,\tilde d_{g_t})$ converges to $(G/H,\tilde d_g)$, which is a length space whose geodesics are $d_{B|_\n}$-Lipschitz (Lemma \ref{lem:Lipschitz}, item $(ii)$). Since $(G/H,d_{B|_\n})$ is induced by a Riemannian metric, Rademacher Theorem (see \cite[Theorem 5.5.7]{burago}) guarantees that geodesics in $(G/H,\tilde d_g)$ are $C^1$ in  a full measure subset of $[0,1]$. 

To conclude the proof of Theorem \ref{thm:collapse}, recall that the length $\ell_{\tilde d_g}$ can be expressed in terms of the \textit{dilatation of $\tilde d_g$}:
\begin{equation*}\label{eq:length}
\ell_{\tilde d_g}(c)=\int_0^1\rm{dil}_{t}(c)dt,
\end{equation*}
where
\begin{equation*}\label{eq:dil}
\rm{dil}_{t}(c)=\limsup_{\epsilon\to 0}\frac{\tilde d_g(c(t-\epsilon),c(t+\epsilon))}{2\epsilon}
\end{equation*}
(see \cite[section 1.1]{gromov}). Since $\ell_{\tilde d_g}$  is given by an integral, it is sufficient to show that the length of the $C^1$ part of  minimizing geodesics is given by the Finsler norm $F:T(G/H)\to \bb R$,
\begin{equation*}\label{eq:Fminimum}
	F(X)=\inf\left\{|Y|_g~:~ Y\in T(G/K),~d\pi(Y)=X  \right\}.
\end{equation*}
Its corresponding distance in $G/H$ is given by
$$
	    d_F(pH, qH) = \inf \left\{ \int_0^1 F(\dot c(t))\, dt:~ c \in C^1([0,1],G/H), \, c(0) = pH, \, c(1) = qH \right\}.
$$
We conclude the proof of Theorem \ref{thm:collapse} with:

\begin{lema}\label{lemma:collapse_part2}
$\tilde d_g(pH, qH) = d_F(pH, qH)$.
\end{lema}
\begin{proof}
    Let $\tilde c: [0,1] \to G/K$  be a $C^1$ curve such that $\tilde c(0) \in pH$, $\tilde c(1) \in qH$. Its projection $c = \pi \circ \tilde c$ is a $C^1$ curve in $G/H$ such that $c(0) = pH$, $c(1) = qH$, thus
    $$
    d_F(pH, qH) \leq \int_0^1 F( \dot c(\xi) )d\xi 
    = \int_0^1F(d\pi(\dot{\tilde c}(\xi)) )d\xi
    \leq \int_0^1 |\tilde c(\xi)|_g d\xi = \ell_g(\tilde c).
    $$
    Therefore, $d_F(pH,qH)\leq \tilde d_g(pH,qH)$.
    On the other hand, for any $\epsilon > 0$, there exists $\xi_\epsilon\in(t-\epsilon,t+\epsilon)$ such that 
$$ 
	\tilde d_g(c(t-\epsilon),c(t+\epsilon)) \leq 
	\label{eq:d1}  \int_{t-\epsilon}^{t+\epsilon}
	|\dot {\tilde c}(\xi)|_{g}~d\xi
	 = 2\epsilon |\dot {\tilde c}(\xi_\epsilon)|_{g}.
$$	
	Which implies that $\rm{dil}_{t}(c)\leq  |\dot {\tilde c}(t)|_{g}$. 
	Since $\tilde c$ is an arbitrary lift of the curve $c$, it follows that $\rm{dil}_{t}(c)\leq F(\dot c(t))$. Thus $\tilde d_g(pH, qH)	 \leq d_F(pH, qH)$.
\end{proof}

\section{Projected Ricci flow}\label{preliminaries} 
%
Given $x = (x_1, \ldots, x_n) \in \mathbb{R}_+^n$ and $\lambda > 0$, $x$ and $\lambda x$ describe essentially the same geometry. 
Hence it is interesting to analyze the Ricci flow up to a change of scale $\lambda$, normalizing it as follows.
	\begin{teorema}
		\label{thm-rescaling}
		For $x \in \R^n_+$, let $R(x)$ be a vector field, homogeneous of degree $0$ in $x$, and $W(x)$ a positive scalar function, homogeneous of degree $\alpha \neq 0$ in $x$. Suppose that $R(x)$ and $\rho(x) = W'(x) R(x)/\alpha$ are of class $C^1$. 
		Then the solutions of
		\begin{equation}
		\label{eq-R1}
		\frac{d x }{dt}= R(x)
		\end{equation}
		can be rescaled in space and positively reparametrized in time to solutions of the normalized flow
		\begin{equation}
		\label{eq-R-normalizada}
		\frac{d x }{dt}= R(x) - \rho(x) x, \qquad W(x) = 1
		\end{equation}
		and vice-versa.
		Furthermore, $R(x) = \lambda x$ with $\lambda \in \R$ and $W(x) = 1$ if, and only if, $x$ is an equilibrium of equation (\ref{eq-R-normalizada}).
	\end{teorema}
	\begin{proof}
		Let $y(t)$ be a solution of (\ref{eq-R1}).  By the homogeneity and positivity of $W$, for each $t$ there exists $\lambda(t) > 0$ such that $x(t) = \lambda(t) y(t)$ satisfies $W(x(t)) = 1$.  Differentiating in $t$ we get (to shorten the notation on this paragraph, we omit $t$ from now on)
		\begin{equation}
		\label{eq-derivada-x}
		x' = \lambda' y + \lambda R(y) 
		\end{equation}
		and
		$$
		0 = W'(x)x' = \lambda' W'(x) y + \lambda W'(x) R(y)  
		$$
		Since $W'(x)$ is homogeneous of degree $\alpha - 1$ on $x$, it follows that
		$$
		\lambda' W'(y) y + \lambda W'(y) R(y) = 0
		$$
		where $W'(y) y = \alpha W(y)$ by Euler's theorem on homogeneous functions. Thus, $\lambda$ satisfies
		$$
		\lambda' = - \frac{W'(y) R(y)}{\alpha W(y)} \lambda,
		\qquad
		W( \lambda(0) y(0) ) = 1
		$$
		which justifies differentiating $\lambda$.  Let $\sigma(y) = \frac{W'(y) R(y)}{\alpha W(y)}$, plugging $\lambda'$ into equation (\ref{eq-derivada-x}) gives
		$$
		x' = \lambda ( R(y) - \sigma(y) y ).
		$$
		Since $R(y)$ is homogeneous of degree $0$ and $\sigma(y)$ is homogeneous of degree $-1$, we can further write
		\begin{equation}
		\label{eq-R-normalizada2}
		x' = \lambda ( R(x) - \sigma(x) x ),
		\end{equation}
		Since $W(x) = 1$, we have that $\sigma(x) = \rho(x)$, concluding that $x$ is a solution to (\ref{eq-R-normalizada}), up to positive time reparametrization.
		
		Reciprocally, let $x(t)$ be a solution of (\ref{eq-R-normalizada}), so that $\rho(x(t)) = \sigma(x(t))$. Take $y(t)$ as the solution of the non-normalized flow (\ref{eq-R1}) with $y(0) = x(0)$, and $\lambda(t)$ as the solution of $\lambda'(t) = - \sigma(y(t)) \lambda(t)$, with $\lambda(0) = 1$.  Define the positive time reparametrization $s$ given by $t = t(s) = \int_0^s \lambda(t) dt$, so that $x(s) = x(t(s))$ satisfies $\frac{d}{ds} x(s) = \lambda(s) \frac{d}{dt} x(s)$. Then $x(s)$ satisfies
		$$
		\frac{d}{ds} x(s) = \lambda(s) (R(x(s)) - \sigma(x(s))x(s) ).
		$$
		By equation (\ref{eq-R-normalizada2}), the last equation is also satisfied by $\lambda(s)y(s)$, with the same initial condition, so it follows that $x(s) = \lambda(s) y(s)$ for all $s$.
		
		An equilibrium of equation (\ref{eq-R-normalizada}) clearly satisfies $R(x) = \lambda x$. Reciprocally, if $x$ satisfies both $R(x) = \lambda x$ and $W(x) = 1$, then $W'(x) R(x) = \lambda W'(x)x = \lambda \alpha W(x) = \lambda \alpha$, by Euler's theorem. Thus $\rho(x) = \lambda$ and $R(x) - \rho(x) x = 0$, as claimed.
	\end{proof}

From now on, consider the homogeneous Ricci flow, with vector field $R(x)$ given by (\ref{eq-ricci-flow-final}), which is a rational, homogeneous function of $x$, 
with degree $0$.  Note that normalizing this flow to unit volume $W(x) = \int_M \vol(x)$ we get the homogeneous gradient Ricci flow \eqref{eq-ricci-flow-normaliz-final}. Indeed, $W$ is positive homogeneous of degree $\alpha = d/2$ and \eqref{eq-derivada-ricci-pontual} implies that $W'(x)R(x) = - S(x)$, since for invariant unit volume metrics the total scalar curvature becomes $T(x) = S(x)$.   It follows that $\rho(x) = - 2 S(x)/d$, for which \eqref{eq-R-normalizada} becomes \eqref{eq-ricci-flow-normaliz-final}, as claimed. 

%

Since the homogeneous gradient Ricci flow \eqref{eq-ricci-flow-normaliz-final} is the gradient flow of the scalar curvature $S(x)$ given in \eqref{eq-scalar-curvature}, it follows that $S(x)$ is strictly decreasing on non-equilibrium unit-volume solutions.  By the next result, the latter property can be recovered for every other normalization of the Ricci flow (\ref{eq-ricci-flow-final}), so that they have a gradient-like behaviour. More precisely,

\begin{proposicao}

	We have that 
	$L(x) = S( x ) \mathrm{vol}(x)^{-d/2}$ is strictly decreasing on non-equilibrium solutions of a normalized flow \eqref{eq-R-normalizada} of the Ricci flow \eqref{eq-ricci-flow-final}.
	In particular, a normalized Ricci flow does not have non-trivial periodic orbits.
\end{proposicao}
\begin{proof}
%
%
Modulo normalization to either $\mathrm{vol}(x) = 1$ or $W(x)=1$, by the previous theorem we have that homogeneous gradient and normalized Ricci flow share the same equilibria.
%
Given a non-equilibrium solution $x(t)$ of a normalized Ricci flow \eqref{eq-R-normalizada}, by Theorem \ref{thm-rescaling} it can be rescaled and positively reparametrized to a solution of \eqref{eq-R1} which can then again be rescaled and positively reparametrized to a non-equilibrium solution $y(t)	$ of the unit-volume Ricci flow.  Since $\vol(x)$ is homogeneous of degree $d/2$ in $x$, the rescaling is given by $y(t) = \mathrm{vol}(x(t))^{-2/d} x(t)$, so that $\mathrm{vol}(y(t)) = 1$. Since $S(x)$ is homogeneous of degree $-1$ in $x$, it follows that
$$
S(y(t)) = \mathrm{vol}(x(t))^{2/d} S(x(t)) = L(x(t))
$$
is strictly decreasing in $t$, since this holds for a positive time reparametrization of $y(t)$.

Now, suppose there exists a non-trivial periodic orbit $x(t)$ with period $T>0$. Then $L(x(t))$ is strictly decreasing with $L(x(0)) < L(x(T)) = L(x(0))$, a contradiction.
\end{proof}

In order to study the limiting behaviour of the Ricci flow, by taking advantage of the rationality of $R(x)$, we normalize it to a simplex and rescale it to get a polynomial vector field.  More precisely, denote by an overline the sum of the coordinates of a vector $x = (x_1, \ldots, x_n) \in \R^n$ 
$$
	\overline{x} = x_1 + \cdots + x_n
$$
Consider then the linear scalar function $W(x) = \overline{x}$, whose level set $W(x) = 1$ in  $\R^n_+$ is the open canonical $n$-dimensional simplex ${\cal T}$ (see Figure \ref{projecao-p-fig}).
%
Note that $\mathcal T$ is a bounded level hypersurface, in contrast with the unbounded unit-volume hypersurface. %

	\begin{figure}[!ht]
		\begin{center}
			\def\svgwidth{4cm}
			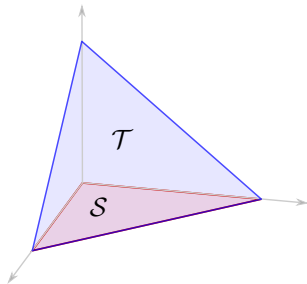
		\end{center}
		\caption{\label{projecao-p-fig} Simplexes ${\cal T}$ and ${\cal S}$ in the case of 3 summands.}
	\end{figure}
	
By linearity we have $W'(x) v = \overline{v}$, so that $\rho(x) = \overline{ R(x)}$, the sum of the coordinates of the vector field $R(x)$.
By the previous results, we get the following.	
	
	\begin{corolario}
		\label{corol-rescaling}
		The solutions of the Ricci flow
		\begin{equation}
		\label{eq-R}
		\frac{d x }{dt}= R(x)
		\end{equation}
		can be rescaled in space and reparametrized in time to solutions of the normalized flow
		\begin{equation}
		\label{eq-R-projetada}
		\frac{d x }{dt}=  R(x) - \overline{R(x)}x, \qquad \overline{x} = 1
		\end{equation}
		and vice-versa, where $x$ is Einstein with $\overline{x} = 1$ if and only if it is an equilibrium of equation (\ref{eq-R-projetada}).
		
		Moreover, there exists a function which is strictly decreasing on non-equilibrium solutions of the normalized flow \eqref{eq-R-projetada}. In particular, the projected Ricci flow does not have non-trivial periodic orbits.
	\end{corolario}

To study the limiting behavior of \eqref{eq-R-projetada} on ${\cal T}$, 
it is convenient to multiply it by an appropriate positive function $f: \mathbb{R}_+^n \to \mathbb{R}_+$
in order to get a homogeneous polynomial vector field $X(x)$ defined in the closure of ${\cal T}$ and tangent to the boundary of $\mathcal T$, given by
	\begin{eqnarray}
	\label{campo-X}
	X(x) & = & f(x)\left (R(x) - \overline{R(x)}\, x \right ) \label{eq-def-X} \\
	& = & (fR)(x) - \overline{(fR)(x)}\, x \nonumber
	\end{eqnarray}
since $W(x) = \overline{x}$ is linear.  In particular, solutions of the new field in the iterior of $\mathcal T$ are time-reparametrizations of \eqref{eq-R-projetada}. 
Therefore, to get a polynomial vector field $X$, it suffices to choose $f$ such that $(fR)(x) = f(x) R(x)$ is a polynomial vector field.
Moreover, in order for $X$ to be tangent to the boundary of $\mathcal T$, it is sufficient that the $i$-th coordinate of $(fR)(x)$ vanishes whenever the $i$-th coordinate does or, equivalently, that each coordinate hyperplane $\Pi_i = \{ x:~x_i = 0\}$ is invariant by the flow of $fR$.  Given a subset of indexes $I \subseteq \{ 1, \ldots, n \}$, consider the subspace $\Pi_I = \cap_{i \in I} \Pi_i$ and let ${\cal T}_I = \mathrm{cl}({\cal T}) \cap \Pi_I$ be the $I$-th face of the simplex ${\cal T}$. 
Note that ${\cal T}_\varnothing = \mathrm{cl}({\cal T})$.
	
	\begin{proposicao}
		\label{propos-estratific}
		If $fR$ is tangent to each hyperplane $\Pi_i$, then each face ${\cal T}_I$ of ${\cal T}$ is invariant by the flow of $X$.  In particular, $\mathrm{cl}({\cal T})$ is invariant and its vertices are fixed points. 
	\end{proposicao}
	\begin{proof}
		Note that $X$ is both tangent to ${\cal T}$ and to each hyperplane $\Pi_i$.  
		By continuity of the solutions in $t$, the invariance of $\Pi_i$ implies the invariance of each semi-space $x_i > 0$ and $x_i < 0$.  The result then follows by taking intersection of these invariant semi-spaces.
	\end{proof}
	
For simplicity, we make an additional modification on the flow.	Instead of analyzing the dynamics of the flow associated to $X$ restricted to ${\cal T}$, it is more convenient to analyze the dynamics of the projection of X to the simplex
	\[
	{\cal S} = \{(x_1,\ldots,x_{n - 1}) \in \mathbb{R}_+^{n-1} :~ x_1 + \cdots + x_{n - 1} \leq 1\}
	\]
	(see Figure \ref{projecao-p-fig}) associated to the conjugated vector field
	$$
	Y = P \circ X \circ P^{-1}
	$$
	where $P: {\cal T} \to {\cal S}$ is given by the projection
$P(x_1,\ldots,x_{n - 1},x_n) = (x_1,\ldots,x_{n - 1})$
	with inverse
$P^{-1}(x_1,\ldots,x_{n - 1}) = P(x_1,\ldots,x_{n - 1}, 1 - x_1 - \cdots - x_{n-1})$.
	The flow of $Y$ in ${\cal S}$ is the so called {\em projected Ricci flow}.
	
	\begin{proposicao}
		\label{prop:simplex-projection}
		If the vector field $fR$ is polynomial of degree $d$, then the vector fields
		$X$ given by equation \eqref{campo-X} and
		 $Y = P \circ X \circ P^{-1}$
		are polynomial of degree $d+1$ and the associated flows are conjugated.  
		Moreover, $x\in\mathcal T$ is Einstein if and only if $Y( Px ) = 0$.
	\end{proposicao}
	
	\begin{proof}
		Since $X$ and $Y$ are conjugated by the linear map $P$, the same is true for their associated flows. The term $\overline{fR(x)}x$ shows that $X$ has degree $d+1$ and it is immediate that $X$ and $Y$ have the same degree since $P$ and $P^{-1}$ have degree one.  From Proposition \ref{thm-rescaling} it follows that $x \in {\cal T}$ is Einstein if and only if $X(x) = 0$.  Since the kernel of $P$ is the $x_n$ axis and since $Y \circ P = P \circ X$, it follows that $Y( Px ) = 0$ if and only if $X(x)$ is parallel to the $x_n$ axis, hence if and only $X(x) = 0$, since $X(x)$ is tangent to ${\cal T}$.
	\end{proof}
	
	The next lemma is well known and connects symmetries of the flow with symmetries of its invariant sets.
	
	\begin{lema}
		\label{lema-simetrias}
		If $T: \mathbb{R}^n \to \mathbb{R}^n$ commutes with the flow $\Phi^t$ of $R$ for all $t$, then the fixed point set of $T$ is $\Phi^t$-invariant.  In particular, if $T$ is a linear isomorphism that commutes with the vector field $R$, then the fixed point set of $T$ is $\Phi^t$-invariant. 
	\end{lema}
	\begin{proof}
		For the first part, if $T(x) = x$ then $T( \Phi^t(x) ) = \Phi^t( T(x) ) = \Phi^t(x)$, so that $\Phi^t(x)$ belongs to the fixed point set of $T$, as claimed. For the second part, note that the flow of the vector field $T \circ R \circ T^{-1} = R$ is $T \circ \Phi^t \circ T^{-1} = \Phi^t$ and use the first part.
	\end{proof}

\section{Flag manifolds of type II}\label{sec:typeI}

We start our analysis with Type II flag manifolds, listed in Table \ref{tab-typeI}, since it includes two infinite families of $SU(n)$ and $SO(2\ell)$ flag manifolds, 
while Type I consists of finitely many flag manifolds of exceptional Lie groups.

We will denote an invariant metric $g$ by a triple of positive real numbers $(x,y,z) \in \R^3_+$. 

\subsection{$SU(m+n+p)/S(U(m)\times U(n) \times U(p))$} \label{sec-su-n}
Let us now consider the family of generalized flag manifolds $SU(m+n+p)/S(U(m)\times U(n) \times U(p))$, which encompasses $SU(3)/T^2$, since $T^2 = S(U(1)\times U(1) \times U(1))$.
It is well known that the isotropy representation of such family decomposes into 3 irreducible components and these homogeneous manifolds admits 4 invariant Einstein metric (up to scale): 1 Einstein-K\"ahler metric and other 3 non-K\"ahler Einstein, see for instance \cite{kimura}.

The components of the Ricci operator of the invariant metric $g$ are given by (see \cite{sakane})
\begin{eqnarray*}
r_x &=& \frac{1}{2x}+ \frac{mnp}{4mn(m+n+p)}\left( \frac{x}{yz} - \frac{z}{xy} -\frac{y}{xz}  \right) \\
r_y &=& \frac{1}{2y}+ \frac{mnp}{4mp(m+n+p)}\left( \frac{y}{xz} - \frac{x}{yz} -\frac{z}{xy}  \right)\\
r_z &=& \frac{1}{2z}+ \frac{mnp}{4np(m+n+p)}\left( \frac{z}{xy} - \frac{x}{yz} -\frac{y}{xz}  \right)
\end{eqnarray*}
and the corresponding Ricci flow equation 
$$
x' = -2xr_x \qquad y'= -2yr_y \qquad z'= -2zr_z
$$
Now we use the results of Section \ref{preliminaries} in order to study the projection of the system of ordinary differential equations on the plane $x+y+z=1$. More precisely, we will consider the vector field $X = (A,B,C)$, given by
\begin{equation*} \label{campoABC}
\begin{pmatrix}
A \\
B \\
C
\end{pmatrix}
=
\begin{pmatrix}
F \\
G \\
H
\end{pmatrix}
- 
(F+G+H)
\begin{pmatrix}
x \\
y \\
z
\end{pmatrix}
\end{equation*}
where
\begin{eqnarray*}
F(x,y,z)&=&-x \left(p (+x^2-y^2-z^2) + 2(m+n+p) yz \right) \\
G(x,y,z)&=& -y \left(n (-x^2+y^2-z^2) + 2(m+n+p)xz \right) \\
H(x,y,z)&=& -z \left(m (+x^2+y^2-z^2) + 2(m+n+p)xy \right)
\end{eqnarray*}
are obtained from the Ricci vector field by multiplying it by $2xyz(m+n+p)$.
A straightforward computation yields
\begin{eqnarray*}
A(x,y,z)&=& x (m z (-x^2+6 x y-y^2-2 y+z^2)+n y (-x^2+6 x z+y^2-z (z+2))\\
 && +p (x^3-x^2-x (y^2-6 y z+z^2)+(y-z)^2))\\ \\
 B(x,y,z)&=& y (m z (-x^2+x (6 y-2)-y^2+z^2)+n (x^2 (-(y-1))+2 x (3 y-1) z\\
 &&+(y-1) (y^2-z^2))+p x (x^2-y^2+6 y z-z (z+2)))\\ \\
C(x,y,z) &=& z (m (x^2 (-(z-1))+2 x y (3 z-1)+(z-1) (z^2-y^2))\\ 
&&+n y (-x^2+6 x z-2 x+y^2-z^2)+p x (x^2-y^2+6 y z-2 y-z^2)). 
\end{eqnarray*} 

In order to project the vector field $X = (A,B,C)$ to the vector field $Y = (u,v)$ on the 
simplex ${\cal S}$, we take
\begin{eqnarray*}
u(x,y)= A(x,y,1-x-y) \qquad v(x,y)=B(x,y,1-x-y)
\end{eqnarray*}
to get the corresponding {\em projected Ricci flow}
\begin{equation}\label{system-u-v-caso-su}
\left\{
\begin{array}{lll}
u(x,y)&=&-x (2 x-1) (m (4 y-1) (x+y-1)+n y (4 x+4 y-3)+p (x (4 y-1)+(1-2 y)^2)) \\
v(x,y)&=&-y (2 y-1) (m (4 x-1) (x+y-1)+n (y (4 x-1)+(1-2 x)^2)+p x (4 x+4 y-3))
\end{array}
\right.
\end{equation}
Below we compute its singularities and the corresponding eigenvalues $\lambda_1$, $\lambda_2$ of its Jacobian.

\begin{teorema}\label{teo-sing-su-n}
Let us consider the flag manifold $SU(m+n+p)/S(U(m)\times U(n)\times U(p))$, with $m \geq n \geq p >0$, and the corresponding projected Ricci flow equations given by (\ref{system-u-v-caso-su}). \info{ foi trocado K\"ahler e non-K\"ahler}We have

\bigskip

\hspace{-60pt}
\begin{tabular}{|c|c|c|c|c|}
\hline 
Singularity & Type of metric & $\lambda_1$ & $\lambda_2$ & Type of singularity \\ 
\hline 
$O =(0,0)$ & degenerate & $m+p$ & $m+n$ & repeller \\ \hline
$P=(0,1)$ & degenerate & $n+p$ & $m+n$ & repeller \\ \hline
$Q=(1,0)$ & degenerate & $n+p$ & $m+p$ & repeller \\  \hline
$K=(0,\frac{1}{2})$ & degenerate & $-\frac{1}{2}(m+n)$ & $-\frac{1}{2}(m+n)$ & attractor \\ \hline 
$L=(\frac{1}{2},\frac{1}{2})$ & degenerate & $-\frac{1}{2}(n+p)$ & $-\frac{1}{2}(n+p)$ & attractor  \\ \hline
$M=(\frac{1}{2},0)$ & degenerate  & $-\frac{1}{2}(m+p)$ & $-\frac{1}{2}(m+p)$ & attractor \\ \hline
$N=\left(\frac{m+n}{2 (m+n+p)} ,\frac{m+p}{2 (m+n+p)}\right)$ & Einstein non-K\"ahler & $\lambda_1(N)$ & $\lambda_2(N)$ & repeller  \\ \hline 
$R=\left(\frac{m+n}{2 (2 m+n+p)},\frac{m+p}{2 (2 m+n+p)}\right)$ & K\"ahler-Einstein  & $-\frac{m (m+n) (m+p)}{(2 m+n+p)^2}$ & $\frac{(m+n) (m+p)}{2 (2 m+n+p)}$ & hyperbolic saddle  \\ \hline 
$S= \left(\frac{1}{2},\frac{m+p}{2 (m+n+2 p)} \right)$ & K\"ahler-Einstein  & $-\frac{p (m+p) (n+p)}{(m+n+2 p)^2}$ & $\frac{(m+p) (n+p)}{2 (m+n+2 p)}$ & hyperbolic saddle   \\ \hline 
$T=\left(\frac{m+n}{2 (m+2 n+p)},\frac{1}{2}\right)$ & K\"ahler-Einstein & $-\frac{n (m+n) (n+p)}{(m+2 n+p)^2}$ & $\frac{(m+n) (n+p)}{2 (m+2 n+p)}$ & hyperbolic saddle  \\ \hline 
\end{tabular}

\bigskip
 
where
\begin{eqnarray*}
\lambda_1(N)&=&\frac{-\sqrt{(m+n) (m+p) (n+p) \left(m^2 (n+p)+m \left(n^2-6 n p+p^2\right) +n p (n+p)\right)}}{4 (m+n+p)^2}\\ &&+\frac{m^2 (n+p)+m (n+p)^2+n^2 p+n p^2}{4 (m+n+p)^2}
\end{eqnarray*}

\begin{eqnarray*}
\lambda_2(N)&=&\frac{\sqrt{(m+n) (m+p) (n+p) \left(m^2 (n+p)+m \left(n^2-6 n p+p^2\right)+n p (n+p)\right)}}{4 (m+n+p)^2} \\
  &&    + \frac{m^2 (n+p)+m (n+p)^2+n^2 p+n p^2}{4 (m+n+p)^2}
\end{eqnarray*}
\end{teorema}

\begin{figure}[ht]
\centering
\def\svgwidth{11cm}
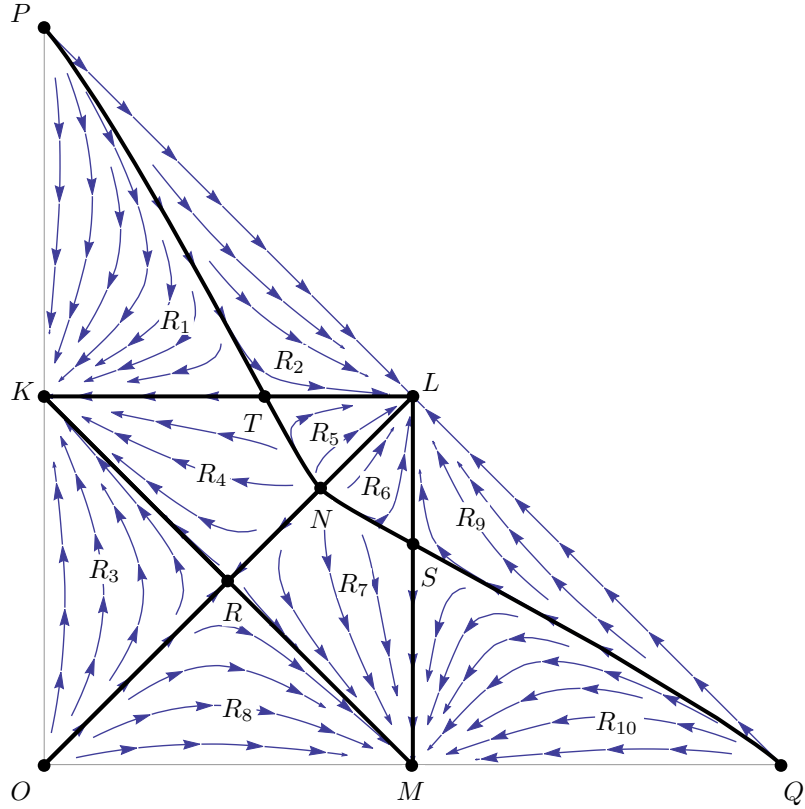
\caption{\label{sing-su-n}
Projected Ricci flow of Type II.}
\end{figure}

\begin{remark}
From Theorem \ref{teo-sing-su-n} one can describe the singularities of the projected Ricci flow equations (\ref{system-u-v-caso-su}) in a very nice way (see Figure \ref{sing-su-n}): it is clear that the singularity $S$ is always in the segment ${LM}$ (supported on the line $x=\frac{1}{2}$), $S$ is always in the segment ${KL}$ (supported on the line $y=\frac{1}{2}$) and $R$ is always in the segment ${KM}$ (supported on the line $x+y=\frac{1}{2}$). Moreover, the point $N$ is always inside the triangle $KLM$. To see this, just note that $\frac{m+n}{2 (m+n+p)}<\frac{1}{2}$, $\frac{m+p}{2 (m+n+p)}<\frac{1}{2}$ and $\frac{m+n}{2 (m+n+p)} +\frac{m+p}{2 (m+n+p)} > \frac{1}{2}$.
\end{remark}

\begin{proposicao}
The segments ${KL}$, ${LM}$, ${MK}$ are invariant by the projected Ricci flow  given by equation \eqref{system-u-v-caso-su}. See Figure \ref{sing-su-n}. 
\end{proposicao}

\begin{proof}
Let us give an explicit proof for the segment $KM$. The other segments follow in a similar way. Since the segment $KM$ is supported by the line $x+y=1/2$, it has $(1,1)$ as a normal vector. The components of the vector field along the line $x+y=1/2$ are given by
\begin{eqnarray*}
u(x,1/2-x)&=& -x (2 x-1) \left(-\frac{1}{2} m \left(4 \left(\frac{1}{2}-x\right)-1\right)+n \left(4 \left(\frac{1}{2}-x\right)+4 x-3\right) \left(\frac{1}{2}-x\right)\right.  \\
&+&\left.   p \left(\left(1-2 \left(\frac{1}{2}-x\right)\right)^2+\left(4 \left(\frac{1}{2}-x\right)-1\right) x\right)\right)
\end{eqnarray*}
\begin{eqnarray*}
v(x,1/2-x)&=& \left(2 \left(\frac{1}{2}-x\right)-1\right) \left(x-\frac{1}{2}\right) \left(-\frac{1}{2} m (4 x-1)   \right. \\
&+& \left. n \left(4 x^2+4 \left(-x-\frac{1}{2}\right) x+x+\frac{1}{2}\right)+p \left(4 \left(\frac{1}{2}-x\right)+4 x-3\right) x\right)
\end{eqnarray*}
A straightforward computation yields
\[(u(x,1/2-x),v(x,1/2-x)) \cdot (1,1)=0\]
and therefore the segment $KM$ is invariant under the flow.
\end{proof}

\begin{exemplo}
Let us consider the flag manifold $SU(4)/S(U(2)\times U(1)\times U(1))$. In this case, we have the following projected Ricci flow
$$
\left\{
\begin{array}{lll}
x' &=& x \left(x^2 (6-32 y)+x \left(-32 y^2+50 y-9\right)+16 y^2-17 y+3\right)\\
y' &=&  -y (2 y-1) \left(16 x^2+x (16 y-17)-3 y+3\right)
\end{array}
\right.
$$
The dynamics of this system is described in Figure \ref{sing-su-n}.
\end{exemplo}

\subsubsection{\label{sec-su-n-gromov}
Gromov-Hausdorff convergence}

We now describe some geometric consequences of the global behavior of the projected Ricci flow by taking into account the phase portrait of the projected Ricci flow (see the regions $R_i$ in Figure \ref{sing-su-n}).
Given an invariant initial metric $g_0$ on the flag manifold  $\mathbb{F}=SU(m+n+p)/S(U(m)\times U(n)\times U(p))$, we now use Theorem \ref{thm:collapse} to understand the metric limit:  \[\lim_{t\to \infty}(\mathbb F,d_{g_t})= (\mathbb{F}_\infty,d)\]

Theorem \ref{thm:collapse} guarantees that the metric limit only depends on the limiting bilinear form $g_i\to g$, therefore $(\bb F_\infty, d)$ is completely determined by the limiting points $K, L, M, O, P, Q$ and the bracket structure of $\lie g$.

Let $\mathfrak{g}$ be the Lie algebra of $SU(m+n+p)$ and consider its reductive decomposition
$\mathfrak{g}=\mathfrak{k}\oplus \mathfrak{m}$.
Recall that the isotropy representation of $\mathbb{F}=SU(m+n+p)/S(U(m)\times U(n)\times U(p))$ decomposes into three irreducible components
$$
\mathfrak{m}=\mathfrak{m}_{1}\oplus \mathfrak{m}_{2} \oplus \mathfrak{m}_{3},
$$
where $\mathfrak m_1=\lie m_{12},\lie m_2=\lie m_{23}$ and $\lie m_{13}=\lie m_3$ are as in \cite{itoh}.  The Lie brackets satisfy
\begin{equation}\label{colchete-su}
\begin{array}{lll}
[\mathfrak{m}_{1},\mathfrak{m}_{1}]\subset \mathfrak{k}, & [\mathfrak{m}_{2},\mathfrak{m}_{2}]\subset \mathfrak{k}, & [\mathfrak{m}_{3},\mathfrak{m}_{3}]\subset \mathfrak{k}, \\

[\mathfrak{m}_{1},\mathfrak{m}_{2}]= \mathfrak{m}_{3} , & [\mathfrak{m}_{1},\mathfrak{m}_{3}]= \mathfrak{m}_{2} , & [\mathfrak{m}_{2},\mathfrak{m}_{3}]= \mathfrak{m}_{1}.
\end{array}
\end{equation}
Recall from Section \ref{sec:gh} that
$$
\m_0=\ker g
\qquad
\h =  \k \oplus \text{Lie algebra generated by }\m_0
$$
A straightforward calculation yields the following.

\begin{lema}
Let $\mathbb{F}=SU(m+n+p)/S(U(m)\times U(n)\times U(p))$ be a flag manifold, and denote by $\mathfrak{g}$ the Lie algebra of $SU(m+n+p)$. Consider the decomposition  $\mathfrak{g}=\mathfrak{m}_{1}\oplus \mathfrak{m}_{2} \oplus \mathfrak{m}_{3}\oplus \mathfrak{k}$. Then the metric limits are as follows
\begin{equation}\label{fibracao1}
	\begin{tabular}{c|c|c|c|c}
		
		Region  &  Limit & $\mathfrak{m}_0$            &  $\mathfrak{h}$   & $G/H$\\
		\hline\hline
		$R_1, R_3,R_4$   &  K     & $\mathfrak{m}_{1}$            & $\mathfrak{k}\oplus\mathfrak m_{1}$  & $Gr_{m+n}(\mathbb{C}^{m+n+p})$\\
		\hline
		$R_2,R_5,R_6,R_9$   &  L     & $\mathfrak{m}_{3}$            & $\mathfrak{k}\oplus\mathfrak m_{3}$  & $Gr_{m+p}(\mathbb{C}^{m+n+p})$\\
		\hline
		$R_7,R_8,R_{10}$   &  M     & $\mathfrak{m}_{2}$            & $\mathfrak{k}\oplus\mathfrak m_{2}$  & $Gr_{n+p}(\mathbb{C}^{m+n+p})$\\
		\hline
		$-R_3,-R_8$   &  O     & $\mathfrak{m}_{1}\oplus \mathfrak m_{2}$ & $\mathfrak g$  & point\\
		\hline
		$-R_1,-R_2$   &  P     & $\mathfrak{m}_{1}\oplus \mathfrak m_{3}$ & $\mathfrak g$  & {point}\\
		\hline
		$-R_9,-R_{10}$   &  Q     & $\mathfrak{m}_{2}\oplus \mathfrak m_{3}$ & $\mathfrak g$  & point\\
		\hline
	\end{tabular}
\end{equation}
where $Gr_s(\mathbb{C}^r)$ represents the Grassmann manifold of $s$-planes inside $\mathbb{C}^r$ with the normal metric and $-R_i$ stands for the backwards projected flow starting in the region $R_i$.
\end{lema}
\begin{proof}
	We are interested in investigating the limiting (sub-Riemannian) metric at each point.  Explicitly, we have (see Figure \ref{sing-su-n})
	\begin{equation}\label{table-degen1}
	\begin{array}{c|c}
	\text{Singularity} & \text{Corresponding degenerate metric}\\ \hline\hline
	K = (0,\frac{1}{2}) & (0, \frac{1}{2}, \frac{1}{2})\\ \hline
	L = (\frac{1}{2},\frac{1}{2} ) & (\frac{1}{2}, \frac{1}{2}, 0)\\\hline 
	M = (\frac{1}{2}, 0) & (\frac{1}{2}, 0, \frac{1}{2})\\ \hline
	O = (0, 0) & (0, 0, 1)\\ \hline
	P = (0,1) & (0, 1, 0)\\ \hline
	Q = (1,0 ) & (1, 0, 0)\\ \hline
	\end{array}
	\end{equation}
	
The Lemma follows by a direct computation using \eqref{colchete-su} (recalling that the bracket of  $\Ad_G(K)$-invariant subspaces is again $\Ad_G(K)$-invariant),  observing that 
the limiting metric is normal homogeneous (i.e., all multiplying factors in \eqref{eq-compon-metr} coincide).
\end{proof}

There is a simple  geometric interpretation for the collapses under the light of Theorem \ref{thm:collapse} as follows.

The first three rows of (\ref{fibracao1}) can be represented as a homogeneous fibrations
\begin{equation}\label{homo-fibracao1}
H/K\to \mathbb{F}=G/K \to G/H
\end{equation}
where the $\mathfrak{m}_{i}$ component of $\mathfrak{h}$ is tangent to the fiber and the other two remaining components of $\mathfrak m$ can be seem both as the \textit{horizontal space} (i.e., the space orthogonal to the fibers) of the fibration or as the tangent to the base. One then has a Riemannian submersion where the limit is given by shrinking its fibers.  Moreover, the fibration (\ref{homo-fibracao1}) has an intuitive geometric interpretation: for instance, the second row is recovered by recalling that $SU(m+n+p)/S(U(m)\times U(n)\times U(p))$ is the manifold of flags of the form $\{ 0\subset V^p\subset V^{n+p}\subset \mathbb{C}^{m+n+p} \}$. Therefore the fibration (\ref{homo-fibracao1}) is just the projection of a flag on a corresponding subspace. For instance, for the second row of (\ref{fibracao1}) we have the projection 
$$\{ 0\subset V^p\subset V^{n+p}\subset \mathbb{C}^{m+n+p} \} \mapsto V^{n+p}.$$
For the third row in table (\ref{fibracao1}), we consider the equivalent (diffeomorphic) flag manifold $SU(m+n+p)/S(U(m)\times U(p)\times U(n))$ and so on. \\

As for the last three rows of (\ref{fibracao1}), $\mathfrak m_{i}\oplus \mathfrak m_{j}$ can be seen as the horizontal space of the corresponding Riemannian submersion of the previous paragraph. In these cases, however,  the reverse flow  shrinks the base, instead of the fibers. Since the horizontal space is completely non-integrable (i.e., its iterated bracket generates the full tangent space), the whole $G/K$ collapses.

We conclude:

\begin{teorema}\label{teo-ricci-su-1}
Consider the flag manifold $\mathbb{F} = SU(m+n+p)/S(U(m)\times U(n)\times U(p))$. Then the limiting behavior of the projected Ricci flow  is given by Figure  \ref{sing-su-n}. In particular \info{alterado K\"ahler e non-K\"ahler }
\begin{enumerate}
\item the K\"ahler Einstein metrics ($R$, $S$ and $T$) are hyperbolic saddles,
\item the non-K\"ahler Einstein metric ($N$) is a repeller,
\item if the metric $g_0$ belongs to $R_1$, $R_3$ or $R_4$ then $\mathbb{F}_\infty= (Gr_{m+n}(\mathbb{C}^{m+n+p}), g_{\rm normal})$,
\item if the metric $g_0$ belongs to $R_2$, $R_5$, $R_6$ or $R_9$ then $\mathbb{F}_\infty= (Gr_{m+p}(\mathbb{C}^{m+n+p}), g_{\rm normal})$,
\item if the metric $g_0$ belongs to $R_7$, $R_8$ or $R_{10}$ then $\mathbb{F}_\infty= (Gr_{n+p}(\mathbb{C}^{m+n+p}), g_{\rm normal})$,
\item if the metric $g_0$ lies outside the triangle delimited by $L$, $T$, $K$, $R$, $M$ and the flow lines connecting them, then $\mathbb{F}_{-\infty}=$ point,
\end{enumerate}
 where $\mathbb{F}_{\pm \infty} = \displaystyle\lim_{t\to \pm\infty} (\mathbb{F}, g_t)$, $g_t$ is the projected Ricci flow with initial condition $g_0$ and the convergence is in Gromov-Hausdorff sense.
\end{teorema}

\subsection{$SO(2\ell)/U(1)\times U(\ell-1)$, $\ell \geq 4$.} \label{sec-so-L}
In this section we will discuss the case of the flag manifold of type $D_\ell$ with tree isotropy summands, namely $SO(2\ell)/U(1)\times U(\ell-1)$, $\ell \geq 4$. The isotropy representation of this flag manifold decomposes into three irreducibles submodules $\mathfrak{m}_1$, $\mathfrak{m}_2$, $\mathfrak{m}_3$ with dimensions $2(\ell-1)$, $2(\ell-1)$ and $(\ell-1)(\ell-2)$, respectively.

By \cite{itoh}, the Lie bracket between the isotropy summands are given by
\begin{equation}\label{colchete-SO2L-U1-UL1}
\begin{array}{lll}
[\mathfrak{m}_{1},\mathfrak{m}_{1}]\subset \mathfrak{k}, & [\mathfrak{m}_{2},\mathfrak{m}_{2}]\subset \mathfrak{k}, & [\mathfrak{m}_{3},\mathfrak{m}_{3}]\subset \mathfrak{k}, \\

[\mathfrak{m}_{1},\mathfrak{m}_{2}]= \mathfrak{m}_{3} , & [\mathfrak{m}_{1},\mathfrak{m}_{3}]=\mathfrak{m}_{2}, &
[\mathfrak{m}_{2},\mathfrak{m}_{3}]=\mathfrak{m}_{1} . 
\end{array}
\end{equation}

Each element in this family of flag manifolds admits 4 invariant Einstein metrics (up to scale): three of them are Einstein-K\"ahler metric and the other one is non-K\"ahler (see \cite{kimura} for details). 

It is worth pointing out that this is exact the same number (and type) of invariant Einstein metrics in the family $SU(m+n+p)/S(U(m)\times U(n) \times U(p))$ (see \ref{sec-su-n}). As we will see in this section, the global behavior of the dynamical system associated to the Ricci flow for flags of $SO(2\ell)$ is also similar to the one described for flags of $SU(n)$. 

Since the computations are very similar to the previous sections we will omit some details. As before, we denote an invariant metric $g$ by the triple of positive real numbers $(x,y,z)$. The components of the Ricci operator of the invariant metric $g$ can be computed by the methods in \cite{anastassiou-chrysikos} and are given by
\begin{eqnarray*}
r_x&=&\frac{(\ell-2)}{8 (\ell-1)}  \left( +\frac{x}{y z}
-\frac{y}{x z}-\frac{z}{x y}\right) +\frac{1}{2 x} \\ \\
r_y&=& \frac{(\ell-2)}{8 (\ell-1)} \left(-\frac{x}{y z}+\frac{y}{x z}-\frac{z}{x y}\right) +\frac{1}{2 y}\\ \\
r_z&=& \frac{1}{4 (\ell-1)}  \left( -\frac{x}{y z}-\frac{y}{x z}+\frac{z}{x y}  \right)+\frac{1}{2 z}
\end{eqnarray*}
For the projected Ricci flow, we proceed as in the previous sections. We start with the auxiliary functions $F, G, H$ given by 
\begin{eqnarray*}
F(x,y,z) &=& -x \left((\ell-2)(+x^2-y^2-z^2) +4 (\ell-1)yz\right)
\\ 
G(x,y,z) &=& -y \left((\ell-2)(-x^2+y^2-z^2) + 4 (\ell-1) xz\right) \\
H(x,y,z) &=& -z \left(\phantom{\ell-2\,}2(-x^2-y^2 + z^2) +4(\ell -1) x y\right)
\end{eqnarray*}

Computing the vector field $(A,B,C)$ determined by Equation (\ref{campoABC}), we get
\begin{eqnarray*}
A(x,y,z)&=& x \left((\ell-2) x^3-x^2 (\ell y+\ell-2 y+2 z-2)-x \left((\ell-2) y^2-12 (\ell-1) y z+(\ell-2) z^2\right) \right. \\
&&\left. +(\ell-2) y^3+y^2 (\ell-2 (z+1))+y z (2 (z+2)-\ell (z+4))+z^2 (\ell+2 z-2)\right) \\ \\
B(x,y,z) &=& y \left((\ell-2) x^3+x^2 (\ell (-y)+\ell+2 y-2 z-2)-x \left((\ell-2) y^2-12 (\ell-1) y z \right. \right. \\ 
&& \left. \left. +\ell z (z+4)-2 z (z+2)\right)+\left(y^2-z^2\right) (\ell (y-1)-2 (y+z-1))\right)\\ \\
C(x,y,z) &=& z \left((\ell-2) x^3+x^2 (-(\ell-2) y-2 z+2)-x \left((\ell-2) y^2-4 (\ell-1) y (3 z-1) \right. \right.  \\
&&\left. \left. +(\ell-2) z^2\right) \left  (y^2-z^2\right) ((\ell-2) y-2 z+2)\right). 
\end{eqnarray*}
We then get the corresponding {\em projected Ricci flow}
\begin{equation} \label{system-u-v-caso-so}
\left\{
\begin{array}{lll}
u(x,y)&=&-x (2 x-1) \left(\ell \left(x (8 y-1)+8 y^2-7 y+1\right)-4 y (2 x+2 y-1)\right)\\
v(x,y)&=&-y (2 y-1) \left(\ell \left(8 x^2+x (8 y-7)-y+1\right)-4 x (2 x+2 y-1)\right)
\end{array}
\right.
\end{equation}
For the result below we computed its singularities and the corresponding eigenvalues $\lambda_1$, $\lambda_2$ of its Jacobian.

\begin{teorema}\label{teo-sing-so}
Consider the flag manifold $SO(2\ell)/(U(1)\times U(\ell-1))$, $\ell \geq 4$, and its corresponding projected Ricci flow equations (\ref{system-u-v-caso-so}). We have \info{foi alterado Kahler e non Kahler}

\bigskip

\hspace{-36pt}
\begin{tabular}{|c|c|c|c|c|}
\hline 
Singularity & Type of metric & $\lambda_1$ & $\lambda_2$ & Type of singularity \\ 
\hline 
$O=(0,0)$ & degenerate & $\ell$ & $\ell$ & repeller \\ \hline
$P=(0,1)$ & degenerate & $\ell$ & $\ell$ & repeller \\ \hline
$Q=(1,0)$ & degenerate & $\ell$ & $2(\ell-2)$ & repeller \\  \hline
$K=(0,\frac{1}{2})$ & degenerate & $-\frac{\ell}{2}$ & $-\frac{\ell}{2}$ & attractor \\ \hline 
$L=(\frac{1}{2},\frac{1}{2})$ & degenerate & $2-\ell$ & $2-\ell$ & attractor \\ \hline
$M=(\frac{1}{2},0)$ & degenerate & $-\frac{\ell}{2}$ & $-\frac{\ell}{2}$ & attractor \\ \hline
$N=( \frac{\ell}{4 (\ell-1)},\frac{\ell}{4 (\ell-1)})$  &  Einstein non-K\"ahler & $\frac{(\ell-2) \ell}{2 (\ell-1)^2}$ & $\frac{(\ell-2)^2 \ell}{4 (\ell-1)^2}$ & repeller \\ \hline 
$R=(\frac{1}{4},\frac{1}{4})$ & K\"ahler Einstein & $-\frac{1}{2}$ & $\frac{\ell}{4}$ & hyperbolic saddle \\ \hline 
$S=(\frac{\ell}{6 \ell-8},\frac{1}{2})$ & K\"ahler Einstein  & $\lambda_1(S) $ & $\lambda_2(S)$ & hyperbolic saddle \\ \hline 
$T=(\frac{1}{2}, \frac{\ell}{6 \ell-8})$ & K\"ahler Einstein  & $\lambda_1(T)$ & $\lambda_2(T)$ & hyperbolic saddle  \\ \hline 
\end{tabular} 

\bigskip

where (in decimal approximation)
\begin{align*}
\lambda_1(S)=\lambda_1(T)~=&~\frac{\ell}{(1.33333\, - \ell)^2} \left((0.0555556 \ell-0.111111) \ell \right. \\
& \left. -0.5 \sqrt{\ell (\ell ((0.308642 \ell-2.22222) \ell+5.97531)-7.11111)+3.16049}\right)
\end{align*}

\begin{align*}
\lambda_2(S)=\lambda_2(T)~=&~\frac{\ell}{(1.33333\, - \ell)^2} \left((0.0555556 \ell-0.111111) \ell \right. \\
& \left. +0.5 \sqrt{\ell (\ell ((0.308642 \ell-2.22222) \ell+5.97531)-7.11111)+3.16049}\right)
\end{align*}

\end{teorema}

\begin{remark}
The phase portrait of Type II $SO(2\ell)$-flags  is very similar to the one obtained for $SU(n)$-flags. See Figure \ref{sing-su-n}.
\end{remark}

\begin{exemplo}
Let us consider the flag manifold $SO(12)/U(1)\times U(5)$. In this case, we have the following projected Ricci flow
$$
\left\{
\begin{array}{lll}
x' &=& -x (2 x-1) \left(6 \left(x (8 y-1)+8 y^2-7 y+1\right)-4 y (2 x+2 y-1)\right) 	\\
y' &=&  -y (2 y-1) \left(6 \left(8 x^2+x (8 y-7)-y+1\right)-4 x (2 x+2 y-1)\right)
\end{array}
\right.
$$
\end{exemplo}


\subsubsection{Gromov-Hausdorff convergence}
Analogously to  $SU(m+n+p)/S(U(m)\times U(n) \times U(p))$, we have
\begin{lema}
Let $G/H=SO(2\ell)/U(1)\times U(\ell-1)$, $\ell \geq 4$ be a flag manifold and denote by $\mathfrak{g}$  the Lie algebra of $SO(2\ell)$. Consider the reductive decomposition $\mathfrak{g}=\mathfrak{k} \oplus \mathfrak{m}_{1}\oplus \mathfrak{m}_{2} \oplus \mathfrak{m}_{3}$. Then the limiting behavior of the projected Ricci flow is given by:
\begin{equation} \label{fibracao-so}
\begin{tabular}{c|c|c|c|c}
Region  &  Limit & $\mathfrak{m}_0$            &  $\mathfrak{h}$   & $G/H$\\
\hline\hline
$R_1, R_3,R_4$   &  K     & $\mathfrak{m}_{1}$            & $\mathfrak m_{1}\oplus \mathfrak{k}$  & $SO(2\ell)/U(\ell)$\\
\hline
$R_7,R_8,R_{10}$   &  M     & $\mathfrak{m}_{2}$            & $\mathfrak m_{2}\oplus \mathfrak{k}$  & $SO(2\ell)/U(\ell)$\\
\hline
$R_2,R_5,R_6,R_9$   &  L     & $\mathfrak{m}_{3}$            & $\mathfrak m_{3}\oplus \mathfrak{k}$  & $SO(2\ell)/(SO(2\ell-2)\times SO(2))$\\
\hline
$-R_3,-R_8$   &  O     & $\mathfrak{m}_{1}\oplus \mathfrak m_{2}$ & $\mathfrak g$  & point\\
\hline
$-R_1,-R_2$   &  P     & $\mathfrak{m}_{1}\oplus \mathfrak m_{3}$ & $\mathfrak g$  & point\\
\hline
$-R_9,-R_{10}$   &  Q     & $\mathfrak{m}_{2}\oplus \mathfrak m_{3}$ & $\mathfrak g$  & point\\
\hline
\end{tabular}
%
\end{equation}
where $SO(2\ell)/U(\ell)$ is the space of orthogonal complex structure on $\mathbb{R}^{2\ell}$ and $SO(2\ell)/SO(2\ell-2)\times SO(2)$ is the Grassmannian of oriented real $2$-dimensional subspaces of $\mathbb{R}^{2\ell}$, both with normal metrics.
\end{lema}

Since the projected Ricci flow of $SO(2\ell)/U(1)\times U(\ell-1)$, $\ell \geq 4$ and $SU(m+n+p)/S(U(m)\times U(n) \times U(p))$ are equivalent, we keep in mind Figure \ref{sing-su-n} in order to state our result about Gromov-Hausdorff convergence.

\begin{teorema}\label{teo-ricci-so-1}
Consider the flag manifold $\mathbb{F} = SO(2\ell)/U(1)\times U(\ell-1)$, $\ell \geq 4$. Then the limiting behavior of the projected Ricci flow  is given by Figure  \ref{sing-su-n}. In particular \
\begin{enumerate}
\item The Einstein-K\"ahler metrics ($R$, $S$ and $T$) are hyperbolic saddles,
\item The Einstein non-K\"ahler metric ($N$) is a repeller,
\item if the metric $g_0$ belongs to $R_1$, $R_3$ or $R_4$ then $\mathbb{F}_\infty= (SO(2\ell)/U(\ell), g_{\rm normal})$,
\item if the metric $g_0$ belongs to $R_2$, $R_5$, $R_6$ or $R_9$ then $\mathbb{F}_\infty= (SO(2\ell)/SO(2\ell-2)\times SO(2) , g_{\rm normal})$,
\item if the metric $g_0$ belongs to $R_7$, $R_8$ or $R_{10}$ then $\mathbb{F}_\infty= (SO(2\ell)/U(\ell), g_{\rm normal})$,
\item if the metric $g_0$ lies outside the triangle delimited by $L$, $T$, $K$, $R$, $M$ and the flow lines connecting them, then $\mathbb{F}_{-\infty}=$ point,
\end{enumerate}
 where $\mathbb{F}_{\pm\infty} = \displaystyle\lim_{t\to \pm\infty} (\mathbb{F}, g_t)$, $g_t$ is the projected Ricci flow with initial condition $g_0$ and the convergence is in Gromov-Hausdorff sense.
\end{teorema}

\subsection{$E_6/SO(8)\times U(1)\times U(1)$}
Let us consider the flag manifold $E_6/SO(8)\times U(1)\times U(1)$. The Lie algebra of $E_6$ decomposes into $\mathfrak{e}_6=\mathfrak{k}\oplus\mathfrak{m}$, where $\mathfrak{k}$ is the Lie algebra of the isotropy and $[\mathfrak{k},\mathfrak{m}]\subset \mathfrak{m}$ (reductive homogeneous space).

The flag manifolds $E_6/SO(8)\times U(1)\times U(1)$ have three isotropy summands, $\mathfrak{m}_1$,  $\mathfrak{m}_2$,  $\mathfrak{m}_3$, with $\dim\mathfrak{m}_i=16$, $i=1,2,3$, therefore the dimension of this flag manifold is 48.

By \cite{itoh}, the Lie bracket between the isotropy summands are given by
\begin{equation}\label{colchete-E6-SO8U1U1}
\begin{array}{lll}
[\mathfrak{m}_{1},\mathfrak{m}_{1}]\subset \mathfrak{k}, & [\mathfrak{m}_{2},\mathfrak{m}_{2}]\subset \mathfrak{k}, & [\mathfrak{m}_{3},\mathfrak{m}_{3}]\subset \mathfrak{k}, \\

[\mathfrak{m}_{1},\mathfrak{m}_{2}]=\mathfrak{m}_{3} , & [\mathfrak{m}_{1},\mathfrak{m}_{3}]=\mathfrak{m}_{2} , &
[\mathfrak{m}_{2},\mathfrak{m}_{3}]=\mathfrak{m}_{1}.
\end{array}
\end{equation}

The components of the Ricci operator of the invariant metric $g$ can be computed by the methods in \cite{anastassiou-chrysikos} and are given by
\begin{eqnarray*}
r_x&=&\frac{1}{12} \left(+\frac{x}{y z}-\frac{y}{x z}-\frac{z}{x y}\right)+\frac{1}{2 x} \\ \\
r_y&=& \frac{1}{12} \left(-\frac{x}{y z}+\frac{y}{x z}-\frac{z}{x y}\right)+\frac{1}{2 y} \\ \\ 
r_z&=& \frac{1}{12} \left(-\frac{x}{y z}-\frac{y}{x z}+\frac{z}{x y}\right)+\frac{1}{2 z}
\end{eqnarray*}

Coincidentally the expressions of the Ricci tensor for $E_6/SO(8)\times U(1)\times U(1)$ are as in the flag $SU(3)/T^2$ (case $m=n=p=1$ in section \ref{sec-su-n}). The components of the Einstein metrics of these two spaces are the same (see \cite{kimura}). Consequentially, the dynamics of the projected Ricci flow $E_6/SO(8)\times U(1)\times U(1)$ is the same as the flag $SU(3)/T^2$. 



%
%

\subsubsection{Gromov-Hausdorff convergence}
Since the Gromov-Hausdorff limit just depends on the limiting bilinear form (Theorem \ref{thm:collapse}), following table \eqref{fibracao1} it is just left to observe that $(\lie e_6,\lie m_i\oplus \lie k)$ is the symmetric pair corresponding to  $E_6/(SO(10)\times U(1))$, the {\em Complexified Cayley projective plane}.  
Recall that $(\g,\h)$ is called a \textit{symmetric pair} (see \cite{helgason}) if there is a decomposition $\g=\m\oplus\h$ such that
 \begin{equation}\label{eq:symmetricspace-condition}
 [\h,\h]\subset \h,\quad [\h,\m]\subset \m,\quad [\m,\m]\subset \h.
 \end{equation}
Note that the first and second conditions account for $\h$ being a Lie subalgebra  and for $\g=\h\oplus
\m$ being a reductive decomposition.


\begin{teorema}\label{teo-ricci-E6-1}
Consider the flag manifold $\mathbb{F} = E_6/(SO(8)\times U(1)\times U(1))$. Then the limiting behavior of the projected Ricci flow  is given by Figure \ref{sing-su-n}. In particular 
\begin{enumerate}
\item The Einstein K\"ahler metrics ($R$, $S$ and $T$) are hyperbolic saddle points,
\item The Einstein non-K\"ahler metric ($N$) is a repeller,
\item if the metric $g_0$ belongs to $R_1$, $R_3$ or $R_4$ then $\mathbb{F}_\infty= (E_6/SO(10)\times U(1), g_{\rm normal})$,
\item if the metric $g_0$ belongs to $R_2$, $R_5$, $R_6$ or $R_9$ then $\mathbb{F}_\infty= (E_6/SO(10)\times U(1) , g_{\rm normal})$,
\item if the metric $g_0$ belongs to $R_7$, $R_8$ or $R_{10}$ then $\mathbb{F}_\infty= (E_6/SO(10)\times U(1), g_{\rm normal})$,
\item if the metric $g_0$ lies outside the triangle delimited by $L$, $T$, $K$, $R$, $M$ and the flow lines connecting them, then $\mathbb{F}_{-\infty}=$ point,
\end{enumerate}
where $\mathbb{F}_{\pm\infty} = \displaystyle\lim_{t\to \pm\infty} (\mathbb{F}, g_t)$, $g_t$ is the projected Ricci flow with initial condition $g_0$ and the convergence is in Gromov-Hausdorff sense.
\end{teorema}

\subsubsection{Topological equivalence of the flows}\label{top}


As we have a complete description of the Ricci flow for flag manifold with three isotropy summands, we can use the Peixoto's Theorem (see \cite{peixoto})
and construct the homeomorphism that give us the topological equivalence.

\begin{teorema} \label{topo-eqv1}
The dynamics of the projected Ricci flows 
of Type II flag manifolds
are topologically equivalent.
\end{teorema}
\begin{proof}
From the previous results of this section, all Type II projected Ricci flows have the same number of singularities, all of the same type, have no saddle connections and the boundary of invariant regions are limited by trajectories.  They also have no nontrivial periodic orbits (Corollary \ref{corol-rescaling}).
Thus, the conditions for constructing the conjugacy homeomorphism of Theorem 1 in \cite{peixoto} are satisfied.
\end{proof}

\section{Flag manifolds of type I}\label{sec:typeII}

In this section we consider the family of flag manifolds of exceptional Lie groups listed in Table \ref{tab-exep}. According to \cite{kimura} each of these manifolds have 3 isotropy summands and 3 invariant Einstein metrics (one K\"ahler--Einstein and two  non-K\"ahler). Note that the family of flags considering in Section \ref{sec:typeI} has 4 invariant Einstein metric. 

We will provide an analysis of the global behavior of projected Ricci flow in a similar fashion as in Section \ref{sec-su-n}.  Again, we will denote an invariant metric $g$ by a triple of positive real numbers $(x,y,z) \in \R^3_+$. 

Let $G/K$ be a flag manifold in Table \ref{tab-exep} and consider the decomposition of the tangent space at the trivial coset $b = K$ into irreducible components, $\mathfrak{m}=\mathfrak{m}_1\oplus \mathfrak{m}_2\oplus \mathfrak{m}_3$. The dimension $d_i$ of each component $\mathfrak{m}_i$ was computed in \cite{kimura} and is also listed in Table \ref{tab-exep}. The brackets between the isotropy components satisfies (see \cite{itoh})
\begin{equation}
\label{colchete-exep}
\begin{array}{lll}
[\mathfrak{m}_1, \mathfrak{m}_1] \subset \mathfrak{k}\oplus \mathfrak{m}_2, &
[\mathfrak{m}_2, \mathfrak{m}_2] \subset \mathfrak{k}, &
[\mathfrak{m}_3, \mathfrak{m}_3] \subset \mathfrak{k}, \\

[\mathfrak{m}_1, \mathfrak{m}_2] \subset \mathfrak{m}_1\oplus \mathfrak{m}_3, &
[\mathfrak{m}_1, \mathfrak{m}_3] \subset \mathfrak{m}_2, &
[\mathfrak{m}_2, \mathfrak{m}_3] \subset \mathfrak{m}_1.
\end{array}
\end{equation}

Given an invariant metric $g$, it is determined by three positive real number $(x,y,z)$. The components of the Ricci operator for the invariant metric $g$ for the flag manifolds in  Table \ref{tab-exep} were computed in \cite{anastassiou-chrysikos}
\begin{eqnarray*}
r_x & = & \frac{y (-d_1 d_2-2 d_1 d_3+d_2 d_3)}{2 x^2 d_1 (d_1+4 d_2+9 d_3)}+\frac{d_3 (d_1+d_2) }{2 d_1 (d_1+4 d_2+9 d_3)}\left(\frac{x}{y z}-\frac{z}{x y}-\frac{y}{x z}\right)+\frac{1}{2 x}\\ \\
r_y &=& -\frac{(-d_1 d_2-2 d_1 d_3+d_2 d_3)}{4 d_2 (d_1+4 d_2+9 d_3)}\left(\frac{y}{x^2}-\frac{2}{y}\right) +\frac{d_3 (d_1+d_2) }{2 d_2 (d_1+4 d_2+9 d_3)}\left(-\frac{x}{y z}-\frac{z}{x y}+\frac{y}{x z}\right)+\frac{1}{2 y} \\ \\
r_z &=& \frac{(d_1+d_2) }{2 (d_1+4 d_2+9 d_3)}\left(-\frac{x}{y z}+\frac{z}{x y}-\frac{y}{x z}\right)+\frac{1}{2 z}
\end{eqnarray*}
together with the corresponding Ricci flow equation
$$
x' = -2xr_x \qquad y'= -2yr_y \qquad z'= -2zr_z
$$
For the projected Ricci flow, multiply the Ricci vector field by $x^2yz(d_1+4 d_2+9 d_3)d_1 d_2$. We get
\begin{eqnarray*}
F(x,y,z)&=& -4 d_2 x \left(d_1^2 x y z+d_1 d_2 y z (4 x-y)+d_1 d_3 \left(x^3-x \left(y^2-9 y z+z^2\right)-2 y^2 z\right) \right.\\
&&\left. +d_2 d_3 (x-z) \left(x^2+x z-y^2\right)\right)\\ \\
G(x,y,z)&=& -2 d_1 y \left(d_1 d_2 y^2 z-2 d_1 d_3 (x+z) \left(x^2+x z-y^2\right)+8 d_2^2 x^2 z \right. \\
&&\left. -d_2 d_3 \left(2 x^3-20 x^2 z-2 x y^2+2 x z^2+y^2 z\right)\right) \\ \\
H(x,y,z)&=& 4 d_1 d_2 x z \left(d_1 \left(x^2-x y+y^2-z^2\right)+d_2 \left(x^2-4 x y+y^2-z^2\right)-9 d_3 x y\right)
\end{eqnarray*}
Computing the vector field $(A,B,C)$ in Equation (\ref{campoABC}) yields
\begin{eqnarray*}
A(x,y,z) &=& x (4 d_2^2 d_3 (-1 + x) (x - z) (x^2 - y^2 + x z) +  d_1^2 (-4 d_3 y (x + z) (x^2 - y^2 + x z) \\
&& +  2 d_2 z (-2 x^3 + 4 x^2 y + y^3 - 2 x (y + y^2 - z^2))) -  2 d_1 d_2 (-2 d_2 z (-x^3 + 12 x^2 y + y^2 \\
&& + x (-4 y - 2 y^2 + z^2)) + d_3 (-2 x^4 + 2 x^3 (1 + y) + (-4 + y) y^2 z +  2 x^2 (y^2 - 28 y z + z^2)\\
&&  - 2 x (y^3 + y^2 (1 - 2 z) + z^2 - y z (9 + z))))) \\ \\
B(x,y,z) &=& y (4 d_2^2 d_3 x (x - z) (x^2 - y^2 + x z) +  d_1^2 (-4 d_3 (-1 + y) (x + z) (x^2 - y^2 + x z) \\
&& + 2 d_2 z (-2 x^3 + 4 x^2 y + (-1 + y) y^2 - 2 x (y^2 - z^2))) - 2 d_1 d_2 (2 d_2 x z (x^2 + x (4 - 12 y) \\
&& + 2 y^2 - z^2) + d_3 (-2 x^4 + 2 x^3 (-1 + y) + (-1 + y) y^2 z + 2 x^2 (y^2 - 28 y z + z (10 + z)) \\
&&- 2 x (y^3 + z^2 - y z^2 - y^2 (1 + 2 z)))))\\ \\
C(x,y,z) &=& z (4 d_2^2 d_3 x (x - z) (x^2 - y^2 + x z) +    d_1^2 (-4 d_3 y (x + z) (x^2 - y^2 + x z) \\
&& + d_2 (-4 x^3 (-1 + z) + 2 y^3 z + 4 x^2 y (-1 + 2 z) +  4 x (-1 + z) (-y^2 + z^2))) \\
&& - 2 d_1 d_2 (2 d_2 x (4 x y (1 - 3 z) + x^2 (-1 + z) - (-1 + z) z^2 +  y^2 (-1 + 2 z)) \\
&& + d_3 (-2 x^4 + 2 x^3 y + y^3 z +  2 x^2 (y^2 + y (9 - 28 z) + z^2) +  2 x y (-y^2 + 2 y z + z^2))))
\end{eqnarray*}
We then get the corresponding {\em projected Ricci flow}
\begin{equation}\label{proj-ricci-3-sum-exep}
\left\{
\begin{array}{lll}
u(x,y) &=& x (-4 d_2^2 d_3 (2 x^3 (-1 + y) - (-1 + y) y^2 + x^2 (3 - 4 y + 3 y^2) +   x (-1 + 2 y - 4 y^2 + y^3)) \\
&& - 2 d_1^2 (2 d_3 (-1 + y) y (x (-1 + y) + y^2) +  d_2 ((-1 + y) y^3 + x^3 (-4 + 8 y) \\
&&+ 2 x^2 (3 - 9 y + 4 y^2)+ x (-2 + 8 y - 6 y^2 + y^3))) +  2 d_1 d_2 (-2 d_2 ((-1 + y) y^2 \\
&& + 2 x^3 (-1 + 7 y) + x^2 (3 - 22 y + 13 y^2) - x (1 - 7 y + 4 y^2 + y^3)) \\
&&+  d_3 (x^3 (4 - 64 y) + x^2 (-6 + 86 y - 60 y^2) + y^2 (4 - 5 y + y^2) \\
&&+ x (2 - 24 y + 18 y^2 + 5 y^3)))) \\ \\
v(x,y) &=& y (-4 d_2^2 d_3 x (2 x^2 (-1 + y) + (-1 + y) y^2 + x (1 - 2 y + 3 y^2))  \\
&&- 2 d_1^2 (2 d_3 (-1 + y)^2 (x (-1 + y) + y^2) + d_2 ((-1 + y)^2 y^2 + x^3 (-4 + 8 y) \\
&&+ 2 x^2 (3 - 8 y + 4 y^2) + x (-2 + 6 y - 5 y^2 + y^3))) + 2 d_1 d_2 (2 d_2 x (1 + x^2 (6 - 14 y)\\
&& - 3 y + y^2 + y^3 + x (-7 + 22 y - 13 y^2)) + d_3 (x^3 (28 - 64 y) + (-1 + y)^2 y^2\\
&& +  x^2 (-26 + 88 y - 60 y^2) + x (2 - 6 y - y^2 + 5 y^3))))
\end{array}
\right.
\end{equation}
For the result below we computed its singularities and the corresponding eigenvalues $\lambda_1$, $\lambda_2$ of its Jacobian.

\begin{teorema}
Consider the flag manifolds of Type I and its corresponding projected Ricci flow equations (\ref{proj-ricci-3-sum-exep}). We have \info{alterado Atrator - Repulsor}
\begin{enumerate}
\item degenerate metrics: $O=(0,0)$, $P=(0,1)$, $Q=(1,0)$ are repellers and $L=(\frac{1}{2}, \frac{1}{2})$, $M=(\frac{1}{2},0)$ are attractors.
\item Einstein-K\"ahler metric: $N=(\frac{1}{6},\frac{1}{3})$ is an attractor.    
\item Einstein non-K\"ahler metrics $R$, $S$ are hyperbolic saddles, they depend on $d_1$, $d_2$ and $d_3$ and are given in the following table (in decimal approximation)
\end{enumerate}
\begin{tabular}{lllll}
Flag Manifold $G/K$  & & $R $ & & $ S $  \\ \hline \hline
$E_8/E_6 \times SU(2) \times U(1)$ & & $(0.46847,0.47077)$ & & $(0.28932,0.26453)$ \\ \hline
$E_8/SU(8) \times U(1)$ && $(0.33648, 0.24145)$ && $(0.39343,0.42039) $  \\ \hline
$E_7/SU(5)\times SU(3) \times U(1)$  && $(0.33218, 0.24367)$ && $(0.39938, 0.42346)$  \\ \hline
$E_7/SU(6)\times SU(2) \times U(1)$  && $(0.44544, 0.45244)$ && $(0.30245,0.25819)$ \\ \hline
$E_6/SU(3)\times SU(3) \times SU(2) \times U(1)$ && $(0.32220, 0.24866)$ && $(0.41388,0.43154)$  \\ \hline
$F_4/SU(3)\times SU(2) \times U(1) $ && $(0.34725,0.23562)$ && $(0.37927 ,0.41362)$  \\ \hline
$G_2/U(2)$ && $(0.21154,0.35427)$  && $(0.46117,0.08619)$
\end{tabular}

\end{teorema}

\begin{exemplo}
Let us illustrate the dynamics of the projected Ricci flow for the manifold $E_8/SU(8) \times U(1)$. In this case we have the following system of ordinary differential equations
$$
\left\{
\begin{array}{rcl}
x'&=&-x \left(4 x^3 (55 y-12)+x^2 \left(210 y^2-370 y+72\right)+x \left(-100 y^2+135 y-24\right)+10 \left(y^2-1\right) y^2\right)\\
y'&=&-y \left(20 x^3 (11 y-5)+2 x^2 \left(105 y^2-178 y+59\right)-27 x \left(2 y^2-3 y+1\right)+10 (y-1)^2 y^2\right)
\end{array}
\right.
$$
whose phase portrait is given in Figure \ref{gen-ricci-flow-exep} as well as its basin of attraction, summarizing the above discussion \info[inline] {a figura foi alterada}. 
\end{exemplo}

%

\subsection{Gromov-Hausdorff convergence}
We proceed to analyze the behavior of the projected Ricci flow near degenerate points, in a similar way as in Section \ref{sec-su-n-gromov}. 
According to the computations above, the global behavior of the  projected Ricci flow for the flag manifolds listed in Table \ref{tab-exep} is given by Figure \ref{gen-ricci-flow-exep}.  

\begin{figure}[ht]
\centering
\def\svgwidth{11cm}
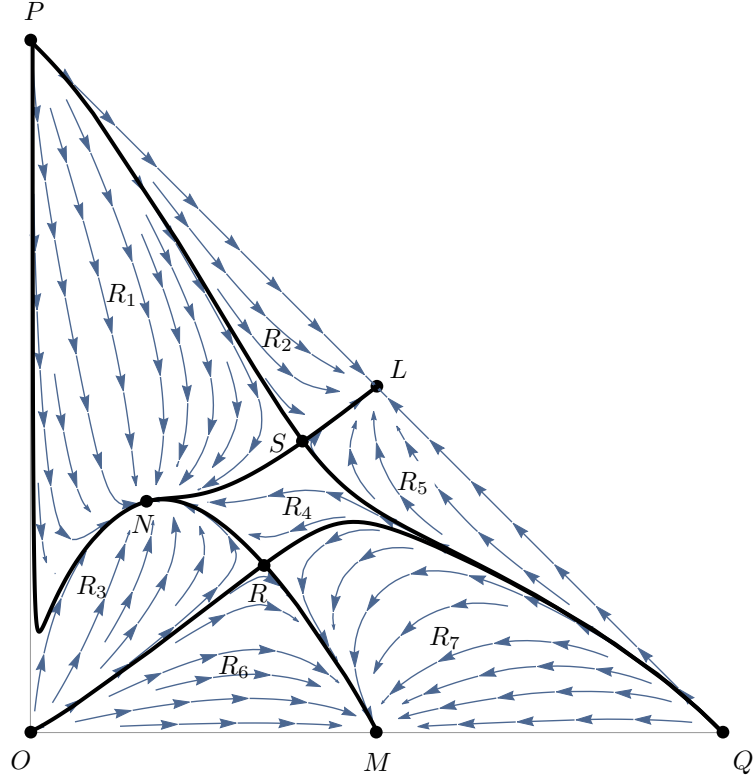
\caption{\label{gen-ricci-flow-exep}
Projected Ricci flow of Type I}
\end{figure}

Let $\mathbb{F}=G/K$ be a generalized flag manifold in Table \ref{tab-exep}. Considering the decomposition $\mathfrak{g}=\mathfrak{k}\oplus \mathfrak{m}_1\oplus\mathfrak{m}_2\oplus\mathfrak{m}_3$, an invariant metric $g$ is determined by the triple $(x,y,z)$ where $x$ correspond to the $\mathfrak{m}_1$-component,  $y$  to the $\mathfrak{m}_2$-component, and $z$ to the $\mathfrak{m}_3$-component.

As in Section \ref{sec-su-n}, the possible Gromov-Hausdorff limits of the flow are determined by the resulting degenerated points,  $M,~ L,~ O,~ P$ and $Q$

	\begin{equation*}\label{table-degen2}
	\begin{array}{c|c}
	\text{Singularity} & \text{Corresponding degenerate metric}\\ \hline\hline
M = (\frac{1}{2}, 0) & (\frac{1}{2}, 0, \frac{1}{2})\\ \hline
L = (\frac{1}{2},\frac{1}{2} ) & (\frac{1}{2}, \frac{1}{2}, 0)\\\hline 
O = (0, 0) & (0, 0, 1)\\ \hline
P = (0,1) & (0, 1, 0)\\ \hline
Q = (1,0 ) & (1, 0, 0)\\ \hline
\end{array}
\end{equation*}

\bigskip

The flow contrasts with the Type II  \info{alterado Backward - Forward} case both in geometric and dynamical aspects: dynamically, it ignores the point $(0,\frac{1}{2})$, geometrically, the
forward flow gives non-symmetric homogeneous spaces \comentario{Llohann}.


We claim that the corresponding Gromov-Hausdorff limits are given by the following table \info{alterado Atrator - Repulsor}:
\begin{equation}\label{table:GH2}
	\begin{tabular}{c|c|c|c|c}
	Region  &  Limit & $\mathfrak{m}_0$            &  $\lie{h}$   & $G/H$\\
	\hline\hline
		-   &  -    & $\lie{m}_1$            & $\g$ & point\\
		\hline
		$R_6, R_7$   &  $M$     & $\lie{m}_2$            & $\lie m_2\oplus \lie{k}$  & Table \ref{table-syme}\\
		\hline
		$R_2, R_5$   &  $L$     & $\lie{m}_3$            & $\lie m_3\oplus \lie{k}$  & Table \ref{tab-BdS}\\
		\hline
		$- R_3, - R_6$   &  $O$     & $\lie{m}_1\oplus \lie m_2$ & $\lie g$  & point\\
	\hline
	$- R_1,- R_2$   &  $P$     & $\lie{m}_1\oplus \lie m_3$ & $\lie g$  & point\\
	\hline
	$- R_4,- R_5,- R_7$   & $Q$     & $\lie{m}_2\oplus \lie m_3$ & $\lie g$  & point\\
	\hline
\end{tabular}
\end{equation}
Since $G$ is semi-simple and $H$ is a subgroup, it follows that $\g=\h^\perp\oplus \h$ is a reductive decomposition of $\g$, where $\h^\perp$ is the $B$-orthogonal complement of $\h$ (note that $\h$ is $\ad_\g(\k)$-invariant, since $\k\subset \h$. Moreover,  $\m_1,\m_2,\m_3$ are pairwise non-isomorphic $\ad_\g(\k)$-representations. Therefore, $\ad_\g(\k)$-invariant subspaces must be the sum of  $\ad_\g(\k)$-irreducible components of $\g$). The class of reductive homogeneous spaces includes the symmetric spaces and the generalized flag manifolds we deal with. Here we restrict to reductive homogeneous spaces  where $G$ is compact simple and  $\rank (H)=\rank(G)$. The classification of such spaces is provided by Borel--de Siebenthal  \cite{borel-sieben} and is an important ingredient in our analysis.

\begin{teorema}[\cite{borel-sieben}]\label{thm:boreldesiebenthal}
 Let $G$ be a compact connected simple Lie group and let $H$ be a 
proper connected subgroup with $\rank(H) = \rank(G)$. Then $G/H$ is either an irreducible inner symmetric space or belongs to the following list:
\begin{equation*}
\begin{array}{lll}
G_2/SU(3), & F_4/(SU(3)\times SU(3)), & E_6/(SU(3)\times SU(3)\times SU(3)), \\
E_7/(SU(6)\times SU(3) ), & E_8/SU(9), & E_8/(E_6\times SU(3)),\\
E_8/(SU(5)\times SU(5)).
\end{array}
\end{equation*}
\end{teorema}

\noindent For a list of irreducible inner symmetric spaces we refer to \cite{besse}. We conclude

\begin{lema}
\label{lemma-symm}\label{lemma-BdS}
Let $\mathbb{F}=G/K$ be a flag manifold of Type I. Then,
\begin{enumerate}
	\item  $(\lie g,\mathfrak{k}\oplus\mathfrak{m}_2)$  is a symmetric pair. The corresponding symmetric space is given in Table \ref{table-syme},
	\item  $(\lie g,\mathfrak{k}\oplus\mathfrak{m}_3)$  is a non-symmetric reductive pair associated with a subgroup $H<G$ with maximal rank. The corresponding reductive homogeneous space is given in Table \ref{tab-BdS}.
\end{enumerate}
\end{lema}
\begin{proof}
Item $(1)$ follows from a direct computation using (\ref{colchete-exep}). For item $(2)$, first observe that $\k\oplus \m_3$ is a subalgebra, thus $\g=(\m_1\oplus\m_2)\oplus(\k\oplus\m_3)$ is a reductive decomposition. Now one can proceed with a case by case analysis and conclude that there is no symmetric space whose dimension coincides with the dimension of $G/H$. 
For instance, one can verify that the table below presents all the possible dimensions realized by the  homogeneous spaces $G/H$ appearing in Theorem \ref{thm:boreldesiebenthal}, where $G$ is explicit in the first line (see  \cite[p.~312--314]{besse}) and \cite{borel-sieben}):
\begin{equation}\label{table:GHdim}
\begin{tabular}{c|c|c|c|c|c}
&  $E_8$ & $E_7$            &  $E_6$   & $F_4$  & $G_2$\\
\hline\hline
Symemtric   &  112, 128     & 54, 64, 70            & 26,32,40,42  & 16, 28 & 8\\
\hline
Non-symmetric   &  200, 162, 168     & 90           & 54  & 36 & 6\\
\hline
\end{tabular}
\end{equation} 
Comparing these values with the values of $d_1+d_2$ in Table \ref{tab-exep}, one concludes that  $(\g,\k\oplus\m_3)$ must be non-symmetric. 
\end{proof}

\begin{table}[h!]
	\caption{\label{table-syme}}
		\begin{tabular}{lllll}
			Flag manifold $G/K$  & & Symmetric space $G/H$ & & $\dim G/H$  \\ \hline \hline
			$E_8/E_6 \times SU(2) \times U(1)$ & & $E_8/(E_7\times SU(2))$  & & 112 \\ \hline
			$E_8/SU(8) \times U(1)$ && $E_8/Spin(16) $ && 128  \\ \hline
			$E_7/SU(5)\times SU(3) \times U(1)$  && $E_7/SU(8)$ && 70  \\ \hline
			$E_7/SU(6)\times SU(2) \times U(1)$  && $E_7/( SO(12)\times SU(2))$ && 64 \\ \hline
			$E_6/SU(3)\times SU(3) \times SU(2) \times U(1)$ && $E_6/(SU(6)\times SU(2))$ && 40 \\ \hline
			$F_4/SU(3)\times SU(2) \times U(1) $ && $F_4/Sp(3)\times SU(2)$ && 28  \\ \hline
			$G_2/U(2)$ && $G_2/SO(4)$ && 8 
		\end{tabular}
\end{table}

\begin{table}[h!]
\caption{\label{tab-BdS}}
\begin{tabular}{lllll}
Flag manifold $G/K$  & & $G/H$ (Borel-de Siebenthal) & & $\dim G/H$  \\ \hline \hline
$E_8/E_6 \times SU(2) \times U(1)$ & & $E_8/(E_6\times SU(3))$ & & 162 \\ \hline
$E_8/SU(8) \times U(1)$ && $E_8/SU(9)$ && 168  \\ \hline
$E_7/SU(5)\times SU(3) \times U(1)$  && $E_7/(SU(6)\times SU(3) )$ && 90  \\ \hline
$E_7/SU(6)\times SU(2) \times U(1)$  && $E_7/(SU(6)\times SU(3) )$ && 90 \\ \hline
$E_6/SU(3)\times SU(3) \times SU(2) \times U(1)$ && $E_6/(SU(3)\times SU(3)\times SU(3))$&& 54 \\ \hline
$F_4/SU(3)\times SU(2) \times U(1) $ && $F_4/(SU(3)\times SU(3))$ && 36  \\ \hline
$G_2/U(2)$ && $G_2/SU(3)$ && 6
\end{tabular}
\end{table}


	For the remaining cases in \eqref{table:GH2}, we claim that neither $\m_1\oplus\m_2\oplus\k$, $\m_1\oplus\m_3\oplus\k$ nor $\m_2\oplus\m_3\oplus\k$  can be subalgebras. Start by supposing that $\h=\m_1\oplus\m_2\oplus\k$ is a subalgebra, so that $\g=\h\oplus\m_3$ is a reductive decomposition. It follows from \eqref{colchete-exep} that $[\m_1\oplus\m_2,\m_3]=0$, in particular $\m_1\oplus\m_2\oplus\k$ is an ideal, contradicting the simplicity of $G$. An immediate consequence is that $[\m_1,\m_2]\supset \m_3$. Using this last fact and analogous arguments, we conclude that $\m_1\oplus\m_3\oplus\k$ and $\m_2\oplus\m_3\oplus\k$ are not subalgebras as well. Therefore, in Table \eqref{table:GH2}, either $\h=\m_i\oplus\k$ or $\h=\g$.

For completeness sake, we also consider the case of $\m_0=\m_1$. According to the last paragraph, either $\m_1\oplus\k$ is a subalgebra or $\h=\g$. If one assumes the subalgebra $\h=\m_1\oplus\k$, then $\h\oplus(\m_2\oplus\m_3)$ must be reductive. In particular $[\m_1,\m_2]\subset \m_3$ and $[\m_1,\m_1]\subset \k$ (both inclusions follow from \eqref{colchete-exep}). Concluding that  $(\g,\m_3\oplus\k)$ is a symmetric pair, contradicting Lemma \ref{lemma-BdS}.

In most cases, the geometric interpretation follows along the same lines as in the case of $SU(m+n+p)/S(U(m)\times U(n)\times U(p))$: for $\m_0=\lie m_2$ or $\m_0=\lie m_3$, the collapse is given along the fibers of the submersion $H/K\cdots G/K\to G/K$. For $\lie m_0=\lie m_1\oplus \lie m_2$ or $\lie m_0=\lie m_1\oplus \lie m_3$, we again have a fibration $G/K\to G/L$ where $\lie l=\lie m_3\oplus \lie k$ or $\lie l=\lie m_2\oplus \lie k$, but the collapse happens in the base, not in the fibers. The collapse of the base forces the collapse of the entire space since horizontal curves (i.e., curves tangent to the distribution defined by $\lie m_0$, in this case) connect every pair of points in  $G/K$. 

More interesting cases are when $\lie m_0=\lie m_1$ and $\lie m_0=\lie m_2\oplus \lie m_3$ since neither $[\lie m_0,\lie m_0]\nsubseteq \lie m_0$ nor $[\lie m_0^\perp,\lie m_0^\perp]\nsubseteq \lie m_0^\perp$, which were the necessary conditions to construct the fibrations above. Thus, the last cases truly expresses the control-theoretic/sub-Riemannian aspect of the collapses which is made clear through Chow-Rascheviskii Theorem (see section \ref{sec:GH1} for details).

Summarizing:
\begin{teorema}
Let $\mathbb{F}=G/K$ be of Type I. Then the limiting behavior of the projected Ricci flow  is given by \info{alterado $+\infty$ e $-\infty$.} Figure  \ref{gen-ricci-flow-exep}. In particular
\begin{enumerate}
\item if $g_0\in R_6$ or $R_7$ then $\mathbb{F}_{\infty}$ is the corresponding symmetric space $G/H$ listed in Table \ref{table-syme}, equipped with the normal metric (up to scale),
\item if $g_0\in R_2$ or $R_5$ then $\mathbb{F}_{\infty}$ is the corresponding Borel-de Siebenthal homogeneous space $G/H$ listed in Table \ref{tab-BdS}, equipped with the normal metric (up to scale),
\item if the metric $g_0$ lies outside the the cusp made up by $L$, $S$, $N$, $M$ and the flow lines connecting them, then $\mathbb{F}_{-\infty}=$ point,
\end{enumerate}
 where $\mathbb{F}_{\pm\infty} = \displaystyle\lim_{t\to \pm\infty} (\mathbb{F}, g_t)$, $g_t$ is the projected Ricci flow with initial condition $g_0$ and the convergence is in Gromov-Hausdorff sense.
\end{teorema}

Following along the same lines as in Theorem \ref{topo-eqv1},  we obtain: 
\begin{teorema}
The dynamics of the projected Ricci flows 
of Type I flag manifolds
are topologically equivalent.
\end{teorema}

\bigskip

\subsubsection*{Data availability}  The {\em Mathematica}\textsuperscript{TM} software code used in the article is available upon request to the authors.


\end{document}